\documentclass[11pt, onepage]{amsart}

\usepackage[usenames,dvipsnames]{pstricks}
\usepackage{epsfig}

\usepackage{amsfonts}
\usepackage{amsmath}
\usepackage{amsthm}
\usepackage{amscd}
\usepackage{euscript}
\usepackage{amssymb}

\usepackage{todonotes}%%% this is to write comments%%%%%%

\usepackage{graphicx}
\usepackage{color}
\usepackage{mathrsfs}
\usepackage{xypic}
\xyoption{all}

\theoremstyle{plain}
\newtheorem{thm}{Theorem}[section]

\newtheorem{lem}[thm]{Lemma}
\newtheorem{prop}[thm]{Proposition}
\theoremstyle{definition}
\newtheorem{defn}{Definition}[section]
\newtheorem{def-prop}[defn]{Definition-Proposition}
\theoremstyle{remark}
\newtheorem{rem}{Remark}[section]
\theoremstyle{example}

\numberwithin{equation}{section}

\newcommand{\Z}{\mathbb Z}

\newcommand{\C}{\mathbb C}

\newcommand{\G}{\mathbb G}
\newcommand{\PP}{{\mathbb P}}

\newcommand{\OO}{\mathcal{O}}

\newcommand{\HHH}{\mathcal{H}}

\newcommand{\W}{\mathcal{W}}

\newcommand{\ul}{\underline}

\newcommand{\OP}{\OO_{\PP^2}}

\newcommand{\tr}[1]{{{}^t#1}}
\newcommand{\ceqrn}{\hat{{\mathcal{C}}}_{r,n}}
\newcommand{\crn}{{\mathcal{C}}_{r,n}}
\newcommand{\crnb}{{\mathcal{S}}^B_{r,n}}
\newcommand{\cnb}{{\mathcal{S}}^B_{n}}
\newcommand{\phib}{\varphi^B}

\newcommand{\sy}{S^2 V}
\newcommand{\sk}{\Lambda^2 V}

\newcommand{\syrc}{S^2V_{\le r}}

\newcommand{\SymplFormMtrx}[1]{
        \begin{bmatrix}
                   0 & I_{#1}\\
                -I_{#1} &   0
        \end{bmatrix}}

\DeclareMathOperator{\di}{\Diamond}
\DeclareMathOperator{\sh}{\sharp}

\DeclareMathOperator{\HH}{H} \DeclareMathOperator{\hh}{h}
 \DeclareMathOperator{\Hom}{Hom}
\DeclareMathOperator{\im}{Im} \DeclareMathOperator{\rk}{rk}
\DeclareMathOperator{\Ker}{Ker} 
\DeclareMathOperator{\coker}{Coker}

\DeclareMathOperator{\T}{T}

\DeclareMathOperator{\Pf}{Pfaff}

 \DeclareMathOperator{\GL}{GL}
 
\DeclareMathOperator{\codim}{codim} \DeclareMathOperator{\SL}{SL}

\DeclareMathOperator{\SO}{SO}
\DeclareMathOperator{\Sp}{Sp}

\begin{document}

\title{Orthogonal bundles and skew-hamiltonian matrices}

\author{Roland Abuaf}
\address{Department of Mathematics, Imperial College London, 180 Queen's Gate, London SW7 2AZ, UK}
\curraddr{}
\email{r.abuaf@imperial.ac.uk}

\author{Ada Boralevi}
\address{Scuola Internazionale Superiore di Studi Avanzati, via Bonomea 265, 34136 Trieste, Italy}
\curraddr{}
\email{ada.boralevi@sissa.it}

\thanks{Research partially supported by the Research Network Program GDRE-GRIFGA. A. Abuaf supported by EPSRC programme grant EP/G06170X/1. 
A.Boralevi supported by SISSA, MIUR funds, PRIN 2010-2011 project 
\lq\lq Geometria delle variet\`a algebriche'', and by Universit\`a degli Studi di Trieste--FRA 2011.} 



% \dedicatory{}

\begin{abstract}
	Using properties of skew-Hamiltonian matrices and classic connectedness results, we prove that the 
	moduli space $M_{ort}^0(r,n)$ of stable rank $r$ orthogonal vector bundles on $\PP^2$, with Chern classes 
	$(c_1,c_2)=(0,n)$, and trivial splitting on the general line, is smooth irreducible of dimension $(r-2)n-{r \choose 2}$ 
	for specific values of $r$ and $n$. 
\end{abstract}

\maketitle

\tableofcontents

\newpage 

%%%%%%%%%%%%%%%%%%%%%%%%%%%%%%%%%%%%%%%%%%%%%%%%%%%%%%%%%%%%%%%%%%%%%%%%%%%%%%%%%
%%%%%%%%%%%%%%%%%%%%%%%%%%%%%%%%%%%%%%%%%%%%%%%%%%%%%%%%%%%%%%%%%%%%%%%%%%%%%%%%%
\section{Introduction}
%%%%%%%%%%%%%%%%%%%%%%%%%%%%%%%%%%%%%%%%%%%%%%%%%%%%%%%%%%%%%%%%%%%%%%%%%%%%%%%%%
%%%%%%%%%%%%%%%%%%%%%%%%%%%%%%%%%%%%%%%%%%%%%%%%%%%%%%%%%%%%%%%%%%%%%%%%%%%%%%%%%

A holomorphic vector bundle on a projective variety is called \emph{orthogonal} if it is isomorphic 
to its dual via a symmetric map. While there is a vast literature about orthogonal bundles on curves (let us quote at least \cite{hulek_orthogonal}, \cite{ramanathan}, \cite{beauville_orthogonal}, and \cite{serman}), very little is known about the case of surfaces. In this work we are interested in the study of stable orthogonal bundles 
on $\PP^2$ with fixed invariants.

In the celebrated paper \cite{Hul}, the author described the moduli space $M(r,n)$ of stable rank $r$ vector bundles on $\PP^2$ with Chern classes 
$(c_1,c_2)=(0,n)$ and $2 \le r \le n$, and proved its smoothness and irreducibility. In \cite{simpl} Ottaviani used Hulek's techniques to 
show that the same properties hold for $M_{sp}(r,n)$, the moduli space of symplectic bundles with the same invariants. 
The generalization is quite straightforward, and the question whether or not these same techniques 
could be applied to the orthogonal case of $M_{ort}(r,n)$ arose naturally, cf. \cite[Problem 7.8]{simpl}. As it turns out, it is definitely not the case.

The smoothness of the moduli space $M_{ort}(r,n)$ is very easy to prove, but the same cannot be said about its (potential) irreducibility. 
A first obstacle is caused by the fact that when deformed on a line, orthogonal bundles behave very differently from their symplectic and unstructured equivalent. 
Indeed, while in the latter two cases the only rigid bundle is the trivial bundle, in the orthogonal case there are two rigid bundles, 
namely the trivial one $\OO_{\PP^1}^r$ and $\OO_{\PP^1}(1) \oplus \OO_{\PP^1}^{r-2} \oplus \OO_{\PP^1}(-1)$, and bundles $\bigoplus_i \OO_{\PP^1}(a_i)$ 
whose value of $\sum_i a_i$ mod $2$ (known as Mumford invariant) is different do not deform into each other. 
This behavior is somehow expected and is connected with the contrasting properties of the group $\SO(r)$ for even and odd values of $r$.

On the one hand, this result forces us to restrict our attention to the moduli space of orthogonal bundles having the same invariants as above and trivial splitting on the general line, that we 
denote by $M^0_{ort}(r,n)$. On the other hand, a careful analysis of this case allows us to extend the notion of Mumford invariant to the case of $\PP^2$.

In the attempt to study irreducibility properties of $M^0_{ort}(r,n)$, we apply techniques that are similar to \cite{Hul} and \cite{simpl}. 
Using standard fibration arguments, we are able to reduce ourselves to the proof of the irreducibility 
of the space $\{(A,B) \in \sk \times \sk\:|\: \rk (AJB-BJA)=r\}$, where $V$ is a complex vector space of even dimension $n$. 

The technical difficulties that this task presents for values of $r$ smaller than $n$ 
are a lot higher than one could expect. By combining a description of the commutator of two skew-Hamiltonian matrices 
together with a strong connectedness result, we further reduce the problem to the estimate of the dimension of the singular locus of highly non-general hyperplane 
sections of the determinantal variety $\syrc$ of symmetric matrices of rank at most $r$. (The secant variety to the Veronese variety if one prefers this terminology.)
This estimate is possible for the cases $r=n$ and $r=n-1$, and gives, respectively, bounds $n \ge 4$ and $n \ge 8$. Our main result is:

\vskip.03in
\textbf{Theorem \ref{teorema 7.7}.} \emph{Let $n$ be an even integer. 
The moduli space $M^0_{ort}(r,n)$ of rank $r$ stable orthogonal vector bundles on $\PP^2$, with Chern classes $(c_1,c_2)=(0,n)$, 
and trivial splitting on the general line, is smooth irreducible of dimension $(r-2)n-{r\choose 2}$ for $r=n$ and $n \ge 4$, and $r=n-1$ and $n\ge 8$.}
\vskip.03in

For small values of $n$ the behavior is even less predictable, as it is explained in section 5. 

\vskip.1in

The paper is structured as follows: in section 2 we introduce the moduli spaces $M(r,n)$, $M_{sp}(r,n)$ and $M_{ort}(r,n)$ of unstructured, symplectic and orthogonal 
stable rank $r$ vector bundles on $\PP^2$ with Chern classes $(c_1,c_2)=(0,n)$ and $2 \le r \le n$. We give the monad construction and prove that $M_{ort}$ is smooth of dimension 
$(r-2)n - {r \choose 2}$. In section 3 we concentrate on the case of those bundles with trivial splitting on the general line, that we denote by $M_{ort}^0(r,n)$: we deduce some interesting consequences of the trivial splitting assumption, describe the key technical difficulty and proceed to give a proof of irreducibility in some specific cases, postponing the proof of the 
key lemma to section 4. There, in section 4, skew-Hamiltonian matrices and their properties are introduced, and the key lemma is proved.
Section 5 contains a detailed description of the image of the map sending a pair of skew-symmetric matrices $(A,B)$ to the symmetric matrix $AJB-BJA$, where $J$ is the standard 
symplectic form. Section 6 is devoted to some explicit examples, open questions and explicit remarks.

\vskip.1in

\textbf{Acknowledgements.} The authors would like to thank Giorgio Ottaviani 
for suggesting the problem and for countless interesting discussions. 
The second named author would like to thank Alex Massarenti, Emilia Mezzetti, and Sofia Tirabassi 
for suggestions and for their infinite patience.

%%%%%%%%%%%%%%%%%%%%%%%%%%%%%%%%%%%%%%%%%%%%%%%%%%%%%%%%%%%%%%%%%%%%%%%%%%%%%%%%%
%%%%%%%%%%%%%%%%%%%%%%%%%%%%%%%%%%%%%%%%%%%%%%%%%%%%%%%%%%%%%%%%%%%%%%%%%%%%%%%%%
\section{The moduli space of orthogonal bundles on $\PP^2$}
%%%%%%%%%%%%%%%%%%%%%%%%%%%%%%%%%%%%%%%%%%%%%%%%%%%%%%%%%%%%%%%%%%%%%%%%%%%%%%%%%
%%%%%%%%%%%%%%%%%%%%%%%%%%%%%%%%%%%%%%%%%%%%%%%%%%%%%%%%%%%%%%%%%%%%%%%%%%%%%%%%%

%%%%%%%%%%%%%%%%%%%%%%%%%%%%%%%%%%%%%%%%%%%%%%%%%%%%%%%%%%%%%%%%%%%%%%%%%%%%%%%%%
\subsection{Notation}
%%%%%%%%%%%%%%%%%%%%%%%%%%%%%%%%%%%%%%%%%%%%%%%%%%%%%%%%%%%%%%%%%%%%%%%%%%%%%%%%%
We work over the field $\C$ of complex numbers. Given a $d$-dimensional 
vector space $W$ over $\C$, we denote by $W^*= \Hom(W,\C)$ its dual, and 
we fix a determinant form so that $W \simeq W^*$. 
The projective space $\PP^{d-1}=\PP (W)$ is the space of lines through $0$, thus $\HH^0(\OO_{\PP W}(1)) = W^*$. 

Given a vector bundle $E$ on $\PP^{d-1}$ we denote by $E(t)$ the tensor product $E \otimes \OO_{\PP^{d-1}}(t)$, for any integer $t$. 

We use lower case letters to denote the dimension of a cohomology group; for any vector bundle $E$ on $\PP^2$, $\hh^i(E):=\dim \HH^i(\PP^2,E)$.

%%%%%%%%%%%%%%%%%%%%%%%%%%%%%%%%%%%%%%%%%%%%%%%%%%%%%%%%%%%%%%%%%%%%%%%%%%%%%%%%%
\subsection{The monad construction}
%%%%%%%%%%%%%%%%%%%%%%%%%%%%%%%%%%%%%%%%%%%%%%%%%%%%%%%%%%%%%%%%%%%%%%%%%%%%%%%%%
Let $M(r,n)$ be the moduli space of stable vector bundles $E$ on $\PP^2$, with rank $r \geq 2$ and Chern classes 
$(c_1(E),c_2(E))=(0,n)$. In \cite[Section 2.1]{Hul} the author proved that  $M(r,n)$ is non-empty if and only if $r \leq n$, 
hence we will always restrict to this case.

\begin{lem}\cite[Lemma 1.1.2]{Hul}\label{h1=n}
If $E$ is an element of $M(r,n)$ then $\chi(E(-i))=-\hh^1(E(-i))=-n$ for $i=1$ and $2$. 
In particular its value is independent from $r$.
\end{lem}

\begin{lem}\cite{Hul}\label{teorema 7.2}
Let $E$ be an element of $M(r,n)$, and set $\PP^2=\PP (U)$. Denote by $V:=\HH^1(E(-1))$, 
which is a vector space of dimension $n$. Then $E$ is the cohomology bundle of the following monad:
\begin{equation}\label{monade}
I \otimes \OO_{\PP^2} \xrightarrow{g} V^* \otimes \Omega_{\PP^2}^1(2) \xrightarrow{f} V \otimes \OO_{\PP^2}(1),
\end{equation}

where $f \in U \otimes V \otimes V$ is the natural multiplication map and $I:=\HH^1(E(-3))$ has dimension $n-r$.
\end{lem}

\begin{proof}
The proof is a standard application of Beilinson Theorem; we give a sketch for the reader's convenience. After Lemma \ref{h1=n}, we can write down the Beilinson 
table of $E$:
$$
\begin{array}{r|ccc}
\HH^2(E(t))&\:0&\:0&\:0\\
\HH^1(E(t))&\:I&\:V^*&\:V\\
\HH^0(E(t))&0&0&0\\
\hline
t&-3&-2&-1
\end{array}
$$

Thus the monad \eqref{monade} is the spectral sequence entailed by Beilinson's result, whose cohomology abutts to $E$. 

The map $f$ is an element of the vector space: 
\begin{align}
	\nonumber \Hom(V^* \otimes \Omega_{\PP^2}^1(2), V \otimes \OO_{\PP^2}(1)) &= V \otimes V \otimes \Hom(\Omega_{\PP^2}^1(2), \OO_{\PP^2}(1))\\
	\nonumber &=V \otimes V \otimes U. \qedhere
\end{align}
\end{proof}

As remarked in \cite[Proposition 7.3]{simpl}, two simple bundles $E(f)$ and $E(f')$ as in Lemma \ref{teorema 7.2} are 
isomorphic if and only if $f$ and $f'$ are $\SL(V)$-equivalent. 

Using the monad (\ref{monade}) we compute that: 
$$
\HH^0(f): V^* \otimes \HH^0(\Omega_{\PP^2}^1(2)) \rightarrow V \otimes \HH^0(\OO_{\PP^2}(1))
$$ 

Now $\HH^0(\Omega_{\PP^2}^1(2))=\Lambda^2 U^* \simeq U$, once the determinant form is fixed, and $\HH^0(\OO_{\PP^2}(1))=U^*$, so the map 
$\HH^0(f)$ is in fact:
\begin{equation}\label{H^0(f)} 
    \HH^0(f): V^* \otimes U\rightarrow V \otimes U^*,
\end{equation}

and it can be identified with the contraction operator that from an element $f \in U \otimes V \otimes V$ induces an element:
$$S_f: V^* \otimes U \rightarrow \Lambda^2U \otimes V \simeq U^* \otimes V$$

through the following steps:
$$
\xymatrix{
U \otimes V^* \ar[d]^{\otimes f} \ar@/_8pc/[dddd]_{S_f}\\
U \otimes V^* \otimes U \otimes V \otimes V \ar[d]^{\hbox{reordering}}\\
(U \otimes U) \otimes (V^* \otimes V) \otimes V  \ar[d]^{\hbox{projection}}\\
\Lambda^2U \otimes (V^* \otimes V) \otimes V \ar[d]^{\hbox{contraction}}\\
\Lambda^2U \otimes V}$$

In particular, if $P,Q$ and $R$ are the three $n \times n$ ``slices'' of $f$, then $S_f=\HH^0(f)$ can be written as:
\begin{equation}\label{matrix rep}
    \HH^0(f)=\left[
\begin{array}{ccc}
0&P&Q\\
-P&0&R\\
-Q&-R&0
\end{array}
\right]
\end{equation}

\begin{lem}\label{rk of H^0(f)}
    With the same notation as above, $\rk \HH^0(f)=2n+r$.
\end{lem}

\begin{proof} The same proof as in \cite{simpl} and \cite[Lemma 1.3]{Hul} applies. Again, we give a sketch for the reader's convenience. 
	From the monad (\ref{monade}) the kernel of $\HH^0(f)$ 
contains the vector space $I$ of dimension $n-r$, hence:
$$\rk \HH^0(f) \leq 3n -(n-r)=2n+r.$$
Looking at the display associated to (\ref{monade}):

$$\xymatrix{&&0 \ar[d]&0\ar[d]&\\
0\ar[r]&I \otimes \OO_{\PP^2} \ar@{=}[d]\ar[r]& \Ker(f) \ar[d]\ar[r]& E \ar[d]\ar[r]&0\\
0\ar[r]&I \otimes \OO_{\PP^2} \ar[r]^-{g}& V^* \otimes \Omega_{\PP^2}^1(2)\ar[d]^f \ar[r]& \coker(g) \ar[d]\ar[r]&0\\
&&V \otimes \OO_{\PP^2}(1) \ar[d]\ar@{=}[r]& V \otimes \OO_{\PP^2}(1)\ar[d]&\\
&&0&0&}$$

we get the cohomology sequence for the bundle $E$. If $\rk \HH^0(f) < 2n+r$, then $\dim(\HH^0(\Ker f))> \dim(I)=n-r$ and thus 
$\HH^0(E) \neq 0$, which is a contradiction because $E$ is stable.
\end{proof}

%%%%%%%%%%%%%%%%%%%%%%%%%%%%%%%%%%%%%%%%%%%%%%%%%%%%%%%%%%%%%%%%%%%%%%%%%%%%%%%%%
\subsection{Unstructured, symplectic and orthogonal bundles}
%%%%%%%%%%%%%%%%%%%%%%%%%%%%%%%%%%%%%%%%%%%%%%%%%%%%%%%%%%%%%%%%%%%%%%%%%%%%%%%%%

\begin{defn} 
A vector bundle $E$ is called \emph{orthogonal} if there is an isomorphism $\alpha: E \rightarrow E^*$ such that 
$\tr{\alpha}=\alpha$. If $\tr{\alpha}=-\alpha$ then $E$ is called \emph{symplectic}. If $E$ is neither orthogonal nor symplectic, it is said to be 
\emph{unstructured}\footnote{One could argue that in order to indicate a bundle that is neither orthogonal nor symplectic, the term ``general'' is preferable to ``unstructured''. 
Nevertheless we chose to reserve the former to indicate---as usual in algebraic geometry---a claim that holds for all elements away from a countable union of Zariski closed subsets in a parameter space.}. 
\end{defn}

\begin{rem}\label{alpha unique} If $E$ is orthogonal then $S^2E$ contains $\OO_{\PP^2}$ as a direct summand. 
	If moreover $E$ is stable, then it is simple and this forces $\HH^0(S^2E)=\C$, so the isomorphism $\alpha$ is unique up 
to scalar. The same remark holds for symplectic bundles, once we substitute the symmetric power $S^2 E$ with the skew-symmetric $\Lambda^2 E$.
\end{rem}

\begin{thm}
The bundle $E(f)$ cohomology of the the monad \eqref{monade} is:
\begin{itemize} 
	\item orthogonal if and only if the map $f \in U \otimes \sk$;
	\item symplectic if and only if the map $f \in U \otimes \sy$.
\end{itemize}
\end{thm}

\begin{proof}
The symplectic case is proved in \cite[Theorem 7.2]{simpl}. For the orthogonal one, we simply generalize the argument. 
A similar statement can also be found (without proof) in \cite[1.7.3]{Hul}. 

Suppose first that the bundle is orthogonal, and that we have an isomorphism $\alpha: E \rightarrow E^*$ such that $\alpha=\tr{\alpha}$. 
Then we define the pairing:
\begin{align}
    \label{pairing} \HH^1(E(-1)) \otimes \HH^1(E(-2)) &\rightarrow \C\\
    \nonumber \:\:\:\:\:\:\:\:\varphi\:\:\:\:\:\: \:\:\:\otimes\:\:\:\:\:\: \:\:\:\psi\:\:\:\:\:\:\:\:\: &\mapsto (\varphi,\psi)_{E(-1)}
\end{align}
as follows. First we recall Serre duality:
\begin{align}
    \nonumber \HH^1(E(-1)) \otimes \HH^1(E^*(-2)) &\rightarrow \C\\
    \nonumber\:\:\:\:\:\:\:\: \varphi\:\:\:\:\:\:\:\: \otimes \:\:\:\:\:\:\:\:\psi^* \:\:\:\:\:\:\:\:&\mapsto <\varphi,\psi^*>_{E(-1)}
\end{align}
which is induced by cup product. Since cup product is skew-commutative in odd dimension, we have that $<\varphi,\psi^*>_{E(-1)}=-<\psi^*,\varphi>_{E(-2)}$.
(For details, see \cite[Prop 1]{barth}.)

Now define the pairing (\ref{pairing}) by setting, with obvious notation:
\begin{equation}\label{def pairing}
(\varphi,\psi)_{E(-1)}:=<\varphi,\alpha(-2)\psi>_{E(-1)}
\end{equation}

Note that the natural multiplication map $f$ is its own adjoint with respect to the pairing (\ref{def pairing}), so if $\alpha$ 
is symmetric $f$ is skew-symmetric (which is our orthogonal case) and conversely in the symplectic case \cite{simpl}. 

The converse uses a similar argument and we omit it.
\end{proof}

In \cite[Section 2.1]{Hul} it is shown that, when non-empty, $M(r,n)$ is a smooth irreducible variety of dimension $2rn-r^2+1$.

Denote by $M_{ort}(r,n)$ (respectively $M_{sp}(r,n)$) the moduli space of orthogonal (resp. symplectic) 
elements of $M(r,n)$. 

In \cite{simpl} the author proved that, when non-empty (in particular when $r$ is even), the space $M_{sp}(r,n)$ is a smooth irreducible variety 
of dimension $(r+2)n-{r+1 \choose 2}$. 

In this work we wish to investigate smoothness and irreducibility properties of the moduli space $M_{ort}(r,n)$.

%%%%%%%%%%%%%%%%%%%%%%%%%%%%%%%%%%%%%%%%%%%%%%%%%%%%%%%%%%%%%%%%%%%%%%%%%%%%%%%%%
\subsection{Smoothness results, degeneration arguments}
%%%%%%%%%%%%%%%%%%%%%%%%%%%%%%%%%%%%%%%%%%%%%%%%%%%%%%%%%%%%%%%%%%%%%%%%%%%%%%%%%

Recall that the adjoint representation for the orthogonal group $\SO(r)$ is isomorphic to the wedge power $\Lambda^2\C^r$. 
Also, any $E \in M_{ort}(r,n)$ is simple, and we have seen in Remark \ref{alpha unique} that $S^2E$ contains $\OO_{\PP^2}$ as a 
direct summand, therefore we must have $\hh^0(\Lambda^2E)=0$. 
By Serre duality we also have $\hh^2(\Lambda^2E)=\hh^0(\Lambda^2E(-3))=0$. 
Hence $\hh^1(\Lambda^2E)=-\chi(\Lambda^2E)$ and applying Hirzebruch-Riemann-Roch formula:
$$\mbox{$\hh^1(\Lambda^2E)=-\chi(\Lambda^2E)=c_2(\Lambda^2E)-{r \choose 2}$}.$$

By applying the splitting principle we compute that $c_2(\Lambda^2E)=(r-2)n$.
It follows that:

\begin{lem}\label{lemma dim} 
	When non-empty, the moduli space $M_{ort}(r,n)$ is smooth of dimension $(r-2)n-{r \choose 2}$.
\end{lem}

Define:
$$M_{ort}^0(r,n):=\{E \in M_{ort}(r,n) \:|\:E_{|_\ell} = \OO_{\PP^1}^r\: \hbox{for some line}\:\ell\}.$$

Notice that by semicontinuity, if $E_{|_\ell}$ is trivial on a line $\ell$, then it is trivial on the general line $\ell$.

\begin{rem}
It is important to underline here that orthogonal bundles behave quite differently 
from their symplectic and unstructured counterparts. In those cases---with obvious notation---one has 
that $\overline{M^0(r,n)}=M(r,n)$ and $\overline{M_{sp}^0(r,n)}=M_{sp}(r,n)$. 
Hence the fact that $M^0(r,n)$ and $M_{sp}^0(r,n)$ are irreducible implies that the same is true for $M(r,n)$ and $M_{sp}(r,n)$. 
Indeed when we restrict a symplectic bundle on $\PP^1$, the only rigid bundle is the trivial one \cite[Section 9.7]{ramanathan}.

In the orthogonal case the situation is more involved. There is no restriction on the parity of the rank, hence 
we can consider both the orthogonal group $\SO(2l+1)$ (type $B_l$) and the group $\SO(2l)$ (type $D_l$). 
A $B_l$-type orthogonal bundle on $\PP^1$ is of the form $\OO \oplus \bigoplus_{i=1}^l \OO(a_i) \oplus \OO(-a_i)$, while $B_l$-type ones 
are $\bigoplus_{i=1}^l \OO(a_i) \oplus \OO(-a_i)$.

In both cases the rigid bundles are the trivial bundle and the bundle $\OO(1) \oplus \OO(-1) \oplus \OO^{2l-1}$ for $B_l$-type, 
or the trivial bundle and $\OO(1) \oplus \OO(-1) \oplus \OO^{2l-2}$ if we are in the $D_l$-type case. (We refer the reader to \cite[Section 9.5]{ramanathan} for details.)
\end{rem}

%%%%%%%%%%%%%%%%%%%%%%%%%%%%%%%%%%%%%%%%%%%%%%%%%%%%%%%%%%%%%%%%%%%%%%%%%%%%%%%%%
%%%%%%%%%%%%%%%%%%%%%%%%%%%%%%%%%%%%%%%%%%%%%%%%%%%%%%%%%%%%%%%%%%%%%%%%%%%%%%%%%
\section{Irreducibility results}
%%%%%%%%%%%%%%%%%%%%%%%%%%%%%%%%%%%%%%%%%%%%%%%%%%%%%%%%%%%%%%%%%%%%%%%%%%%%%%%%%
%%%%%%%%%%%%%%%%%%%%%%%%%%%%%%%%%%%%%%%%%%%%%%%%%%%%%%%%%%%%%%%%%%%%%%%%%%%%%%%%%

From now on we work on $M_{ort}^0(r,n)$. We want to find under which conditions this moduli space is indeed irreducible.

%%%%%%%%%%%%%%%%%%%%%%%%%%%%%%%%%%%%%%%%%%%%%%%%%%%%%%%%%%%%%%%%%%%%%%%%%%%%%%%%%
\subsection{Consequences of trivial splitting}
%%%%%%%%%%%%%%%%%%%%%%%%%%%%%%%%%%%%%%%%%%%%%%%%%%%%%%%%%%%%%%%%%%%%%%%%%%%%%%%%%

The assumption $E \in M_{ort}^0(r,n)$ yields several consequences.

\begin{prop}\label{n even}
If $E \in M_{ort}^0(r,n)$, then $n=c_2(E)$ is even.
\end{prop}

\begin{proof}
We need to introduce the discriminant of a morphism from \cite[1.7.1]{Hul}. Consider again the map 
$f:\HH^1(E(-2)) \otimes \Omega_{\PP^2}^1(2) \rightarrow  \HH^1(E(-1)) \otimes \OO_{\PP^2}(1)$ 
from the defining monad (\ref{monade}). Recall that $f$ can be seen as an element of $V \otimes V \otimes U$, 
hence $f:U^* \otimes V^* \rightarrow V$ and for every $z \in U^*$ we can define a map:
$$f(z): V^* \rightarrow V,\:\:\:\:f(z):=f(- \otimes z).$$

Define the discriminant of $f(z)$ as  $\Delta(f):=\det(f(z))$. Then the following holds:
$$\{z \in U^* \:|\: \Delta(f)=0\} \simeq \{\ell \in {\PP^2}^*\:|\:E_{|_\ell} \neq \OO_{\PP^1}^r\}.$$
To see this, take a line with equation $\{z=0\}$ and tensor its hyperplane sequence by $E(-1)$:
$$0 \rightarrow E(-2) \xrightarrow{z} E(-1) \rightarrow E(-1)_{|_\ell} \rightarrow 0.$$
Taking cohomology, since $\hh^0(E(-2))=\hh^0(E(-1))=0$, we get:
\begin{equation}\label{restrizione retta generica}
    0 \rightarrow \HH^0(E(-1)_{|_\ell}) \rightarrow \HH^1(E(-2)) \xrightarrow{f(z)} \HH^1(E(-1)) \rightarrow \ldots
\end{equation}
Hence $\det(f(z))=0$ if and only if $\hh^0(E(-1)_{|_\ell}) \neq 0$, and this condition is equivalent to $E_{|_\ell} \neq \OO_{\PP^1}^r$. 
If $E$ is orthogonal, then $f(z)$ is skew-symmetric, and we can consider its Pfaffian instead of the determinant.  In order for it to be 
non-zero, one needs $c_2(E)=\hh^1(E(-2))=\hh^1(E(-1))$ even, and this concludes the proof.
\end{proof}

\begin{rem} The discriminant $\Delta(f):=\Pf(f(z)) \in \HH^0(\OO_{{\PP^2}^*}(\frac{n}{2}))$ is a homogeneous polynomial of degree $\frac{n}{2}$ 
(up to a scalar it is uniquely determined by the class $[f]$ in the $\SL(V)$ equivalence). 
Its zero set is a curve of degree $\frac{n}{2}$ in the plane and the proof of Proposition \ref{n even} shows how this curve is related to the splitting behavior of $E$. 
\end{rem}

\begin{rem}\label{mumford invariant} From a result by Mumford \cite[Page 184]{Mumford_theta} it follows that 
if $E$ is an orthogonal bundle on the projective line, then $\hh^0(E(-1))$ mod 2 is invariant under deformations. 
In \cite{hulek_orthogonal} the author proved that orthogonal rank 2 bundles on $\PP^1$ are rigid, while for higher rank the 
Mumford invariant is the only one. More precisely, two orthogonal bundles on $\PP^1$ can be deformed into each other if and only if 
they have the same Mumford invariant. In what follows, one could define $n$ mod 2 (that is, $\hh^1(E(-1))$ mod 2) 
to be the ``Mumford invariant'' for the case of $\PP^2$. (Notice that by Serre duality, on $\PP^1$ one has 
that $\hh^0(E(-1))=\hh^1(E(-1))$.) Proposition \ref{n even} tells us that the parity of $n$ 
is indeed connected with the splitting behavior of $E$ on the general line. 
\end{rem}

\begin{prop}\label{def Z}
If $E \in M_{ort}^0(r,n)$, then $\rk \HH^0(f)=2n +\rk Z$, where $Z:=PQ^{-1}R-RQ^{-1}P$ and $P$, $Q$ and $R$ are as in \eqref{matrix rep}. 
Moreover $\rk Z=r$.
\end{prop}

\begin{proof}
If $E \in M_{ort}^0(r,n)$, then without loss of generality we can assume that any one of the three skew-symmetric slices $P$, $Q$ and $R$ of the map $f$ is invertible.

Just notice that in (\ref{restrizione retta generica}) the map $f(z)$ can be explicitly written as $z_0P+z_1Q+z_2R$. Now 
if the general line has trivial splitting type, taking coordinate lines the map still has to have nonzero Pfaffian, and we 
can assume that the slice $Q$ is invertible. Then we can compute the rank of $\HH^0(f)$ explicitly as:

\begin{align}
    \nonumber \rk \HH^0(f) &= \rk \left[
\begin{array}{ccc}
0&P&Q\\
P&0&R\\
Q&R&0
\end{array}
\right]\\
\nonumber &= \rk\left[
        \begin{array}{ccc}
        I&0&0\\
        0&I&-PQ^{-1}\\
        0&0&I
        \end{array}
        \right]    
        \left[
    \begin{array}{ccc}
    0&P&Q\\
    -P&0&R\\
    -Q&-R&0
    \end{array}
    \right]
    \left[
        \begin{array}{ccc}
        I&0&0\\
        0&I&0\\
        0&-Q^{-1}P&I
        \end{array}
        \right]\\
        \nonumber &= \rk
        \left[
            \begin{array}{ccc}
            0&0&Q\\
            0&Z&R\\
            -Q&-R&0
            \end{array}
            \right]
\end{align} 
   
where $Z=PQ^{-1}R-RQ^{-1}P$. Hence:
\begin{equation}\label{explicit rank}
    \rk \HH^0(f)=2 \rk Q + \rk Z= 2n + \rk Z
\end{equation}

Recall from Lemma \ref{rk of H^0(f)} that $\rk \HH^0(f)=2n +r$ so comparing with \eqref{explicit rank} we deduce that $\rk Z=r$, 
$r$ being the rank of the bundle $E$.
\end{proof}

\begin{rem} When we computed the rank of the map $\HH^0(f)$ in Lemma \ref{rk of H^0(f)}, we did not make any assumption on the matrix $Q$ 
(nor on the splitting type of the bundle). It is easy to show that this rank equals $2n+r$ even if $Q$ is not invertible, 
and there is no contradiction between Proposition \ref{def Z} and Lemma \ref{rk of H^0(f)}.
\end{rem}

\begin{rem} The assumption that the bundle $E(f)$ associated to a monad of type \eqref{monade} has trivial splitting on the general line 
	implies that $f$ corresponds to a semistable point in $\PP(U \otimes \sk)$ under the $\SL(U)\times \SL(V)$-action. 
\end{rem}

%%%%%%%%%%%%%%%%%%%%%%%%%%%%%%%%%%%%%%%%%%%%%%%%%%%%%%%%%%%%%%%%%%%%%%%%%%%%%%%%%
\subsection{Main Theorem}
%%%%%%%%%%%%%%%%%%%%%%%%%%%%%%%%%%%%%%%%%%%%%%%%%%%%%%%%%%%%%%%%%%%%%%%%%%%%%%%%%
We will now use formula (\ref{explicit rank}) to prove the irreducibility of the moduli space $M^0_{ort}(r,n)$. 

Recall that in our setting $n=2p$ is even. There is no loss in generality if we assume that the general invertible skew-symmetric matrix $Q$ is the standard symplectic form 
$J:=\SymplFormMtrx{p}$.
 
Then the matrix $Z$ from Proposition \ref{def Z} is $Z=PJR-RJP$, where again both $P$ and $R$ are skew-symmetric $n \times n$ matrices.

Let us now define:
\begin{equation}\label{def ceqrn}
	\ceqrn:=\{(A,B) \in \sk \times \sk\:|\: \rk (AJB-BJA)=r\}.
\end{equation}

\begin{lem}\label{key lemma}
Let $V$ be a complex vector space of even dimension $n$. The subvariety $\ceqrn$ is irreducible of codimension ${{n-r+1} \choose 2}$ in $\sk \times \sk$ for $r=n$ and $n \ge 4$, 
and for $r=n-1$ and $n \ge 8$.
\end{lem}

Lemma \ref{key lemma} is a key step in our argument, and we will devote next section to its proof. 
Assume it is true for now; then the following holds, which constitutes our main result:

\begin{thm}\label{teorema 7.7}
Let $n$ be an even integer. The moduli space $M^0_{ort}(r,n)$ of rank $r$ stable orthogonal vector bundles on $\PP^2$, with Chern classes $(c_1,c_2)=(0,n)$, 
and trivial splitting on the general line, is smooth irreducible of dimension $(r-2)n-{r\choose 2}$ for $r=n$ and $n \ge 4$, and $r=n-1$ and $n\ge 8$.
\end{thm}

\begin{proof}    
Following \cite[Theorem 1.5.2]{Hul} and \cite[Theorem 7.7]{simpl}, we start by giving a necessary and sufficient condition for an element 
$f \in \PP(U \otimes \sk)$ to give a bundle $E(f)$. Define:
$$K_{r,n}:=\{f \in \PP(U \otimes \sk)\:|\:\rk(\HH^0(f))=2n+r\}.$$

$K_{r,n}$ is quasi-affine, and any $f \in K_{r,n}$ defines $E(f)$ as cohomology bundle of the corresponding monad, once we impose the extra condition that the morphism $V^* \otimes \Omega^1(2) \xrightarrow{f} V \otimes \OO(1)$ is surjective. 

Such maps $f$ form an open subvariety $\tilde{K}_{r,n} \subseteq K_{r,n}$. There is a universal bundle $\mathcal{E}$ over $\PP^2 \times \tilde{K}_{r,n}$ such that 
the fiber $\mathcal{E}_{\PP^2 \times \{ f \}}$ is exactly the bundle $E(f)$, see \cite[Proposition 1.6.1]{Hul}.
Moreover we have an open subvariety $\tilde{K}_{r,n}^s \subseteq \tilde{K}_{r,n}$ consisting of those $f$ giving rise to a stable $E(f)$.
By the universal property of the moduli space, we have a surjection:
$$\xymatrix{\tilde{K}_{r,n}^s \ar@{>>}[r] &M_{ort}(r,n),}$$
and in particular a surjection:
$$\xymatrix{\tilde{K}_{r,n}^s \ar@{>>}[r]^-\pi &M^0_{ort}(r,n).}$$
To prove the theorem it is enough to show that $\pi^{-1}(M^0_{ort}(r,n))$ is irreducible. One has that $\pi^{-1}(M^0_{ort}(r,n))=\tilde{K}_{r,n}^s \setminus Z(\Delta)$, 
and:
$$\mbox{$\tilde{K}_{r,n}^s \setminus Z(\Delta) = \bigcup_{x \in \PP^2} \{ f \in \tilde{K}_{r,n}\:|\:\Delta(f)(x) \neq 0\}=\bigcup_{x \in \PP^2}\tilde{K}_{r,n,x}$}$$
Since any two $\tilde{K}_{r,n,x}$ and $\tilde{K}_{r,n,y}$ have non-empty intersection, we can take advantage of the $\SL(U)$-action: it is enough to prove that 
$\tilde{K}_{r,n,\overline{x}}$ is irreducible for $\overline{x}=(0,1,0)$.
Finally, notice that we have a fibration:
\begin{equation}\label{fibration}
	\tilde{K}_{r,n,\overline{x}} \rightarrow \sk
\end{equation}
sending $f$ to the invertible slice $Q$ of the matrix representation (\ref{matrix rep}),
which is $\SL(V)$ invariant with fibers isomorphic to $\ceqrn$. Irreducibility then follows from the key Lemma \ref{key lemma}.

To conclude the proof, we use the fibration \eqref{fibration} to compute that:
\begin{align}
\nonumber \dim M_{ort}^0(r,n) &= \dim \sk + \dim \ceqrn-\dim\GL(V)\\
\nonumber &=\mbox{${n \choose 2} + {\Big [} 2{n \choose 2} -{{n-r+1} \choose 2} {\Big ]}-n^2 =(r-2)n-{r\choose 2}$},
\end{align}
which agrees with the estimate that we made in Lemma \ref{lemma dim}.
\end{proof}

%%%%%%%%%%%%%%%%%%%%%%%%%%%%%%%%%%%%%%%%%%%%%%%%%%%%%%%%%%%%%%%%%%%%%%%%%%%%%%%%%
%%%%%%%%%%%%%%%%%%%%%%%%%%%%%%%%%%%%%%%%%%%%%%%%%%%%%%%%%%%%%%%%%%%%%%%%%%%%%%%%%
\section{The key lemma}\label{section key lemma}
%%%%%%%%%%%%%%%%%%%%%%%%%%%%%%%%%%%%%%%%%%%%%%%%%%%%%%%%%%%%%%%%%%%%%%%%%%%%%%%%%
%%%%%%%%%%%%%%%%%%%%%%%%%%%%%%%%%%%%%%%%%%%%%%%%%%%%%%%%%%%%%%%%%%%%%%%%%%%%%%%%%

The aim of this section is proving the key Lemma \ref{key lemma}. The reasoning is somewhat 
similar to what is done in the unstructured case treated in \cite{Hul}. There one reduces to prove the irreducibility of pairs of 
$n \times n$ matrices $(A,B)$ whose commutator $[A,B]$ has constant rank $r$. 
This result also has a symmetric analogue proved in \cite{basili,brennan_pinto_vasconcelos}, which is used in \cite{simpl} to show 
irreducibility in the symplectic case. 

The technical difficulty of the skew-symmetric case is however considerably higher than the 
other two cases. In particular the proof of Lemma \ref{C_b irred} requires the use of non-trivial connectedness results.

\vskip.05in

Here are the steps leading to the proof of Lemma \ref{key lemma}.

\begin{enumerate}
	\item We work on the variety: 
	\begin{equation}\label{def crn}
		\crn:=\{(A,B) \in \sk \times \sk\:|\: \rk (AJB-BJA) \le r\},
	\end{equation}
	and prove its irreducibility. The irreducibility of $\ceqrn$ then follows from the fact that $\ceqrn$ is a Zariski open subset of $\crn$.
	
	\item We give a definition of \emph{regular matrix} that works in our setting. We first remark that $\rk(AJB-BJA)=\rk[JA,JB]$, with $[-,-]$ the usual commutator of matrices. 
	This leads us to introduce in the picture skew-Hamiltonian matrices, i.e. matrices of the form $JB$, $B$ skew-symmetric. This is done in Definition \ref{regular}. 
	We then prove in Proposition \ref{dimension lemma} that for a regular matrix $JB$ 
	the kernel of the homomorphism $\phib:\sk \to \sy$ sending a skew-symmetric $A \mapsto AJB-BJA$ has the smallest possible dimension, namely $\frac{n}{2}$.
	
	\item We notice that, for any pair of skew-symmetric matrices $(A,B)$, $AJB-BJA  \in \sy$ is a symmetric matrix.
	
	Fixing a skew-symmetric matrix $B$, we define: 
	$$\crnb:=\{S \in \sy\:|\:S=AJB-BJA \:\:\hbox{for some}\:A \in \sk,\: \rk S \le r\},$$
	and in Lemma \ref{C_b irred} we show that if $JB$ is regular, then $\crnb$ is irreducible of dimension $nr-\frac{3}{2}n-{r \choose 2}$ for $r=n \ge 4$ and 
	for $r=n-1$ and $n \ge 8$.
	
    \item Lemma \ref{open irred} is the second to last step. We  define: 
    \begin{equation}
	\crn^0:=\{(A,B) \in \sk \times \sk\:|\: \:JB\:\hbox{is regular},\:\:\rk (AJB-BJA) \le r\},
\end{equation}
    and we use a fibration argument to deduce the irreducibility of $\crn^0$ from the irreducibility of $\crnb$. 

    \item The last step consists in showing that $\crn$ is the closure of $\crn^0$. This concludes the proof of Lemma \ref{key lemma}, as well as Section \ref{section key lemma}.
\end{enumerate}

%%%%%%%%%%%%%%%%%%%%%%%%%%%%%%%%%%%%%%%%%%%%%%%%%%%%%%%%%%%%%%%%%%%%%%%%%%%%%%%%%
\subsection{Regular skew-Hamiltonians}
%%%%%%%%%%%%%%%%%%%%%%%%%%%%%%%%%%%%%%%%%%%%%%%%%%%%%%%%%%%%%%%%%%%%%%%%%%%%%%%%%

For any pair of (skew-symmetric) matrices $(A,B)$ we make the trivial observation that:
$$\rk (AJB-BJA)=\rk [JA,JB],$$ 
where $[-,-]$ is the usual commutator of matrices. Hence studying symmetric matrices of the form $AJB-BJA$ and fixed (or bounded) rank 
is equivalent to studying the commutator of matrices of the form $JB$, where $B \in \sk$. 
Such matrices are called \emph{skew-Hamiltonian} (or \emph{anti-Hamiltonian}) in literature. Let us recall the following:

\begin{def-prop} Let $W$ and $H$ be elements of $V \otimes V$.
\begin{itemize}
	\item $W$ is called \emph{skew-Hamiltonian} if $JW=\tr{W}J$. $W$ is skew-Hamiltonian if and only $W=JB$, with $B \in \sk$ skew-symmetric. 
	We indicate the space of skew-Hamiltonian matrices by $\W$.
	\item $H$ is called \emph{Hamiltonian} if $JH=-\tr{H}J$. $H$ is Hamiltonian if and only $H=JS$, with $S \in \sy$ skew-symmetric. 
	We indicate the space of skew-Hamiltonian matrices by $\HHH$.
\end{itemize}
$\W$ and $\HHH$ correspond, respectively, to the Jordan algebra and to the Lie algebra of the symplectic group $\Sp(n)$.
\end{def-prop}

We mentioned above that Lemma \ref{key lemma} has an unstructured as well as a symmetric analogue, 
proven respectively in \cite[Proposition 2.3.6]{Hul} and in \cite[Theorem 2.6]{basili} and \cite[Corollary 3.6]{brennan_pinto_vasconcelos}. 
Both arguments make use of regular matrices: 
a \emph{regular matrix} is a regular element of the Lie algebra, and in particular it is an element whose commutator has minimal dimension. 

The proof in the symplectic case is particularly easy: given a symmetric matrix $B$, one defines the linear morphism 
$\varphi^B: \sy \to \sk$, mapping any $A$ to the commutator $[A,B]$. If $B$ is regular the kernel of the morphism is $n$-dimensional, which means 
that $\varphi^B$ is surjective. A standard fibration argument then allows one to conclude irreducibility of the space of 
symmetric matrices whose commutator has fixed rank.

Unfortunately the notion of regular element is meaningless for the Jordan algebra $\W$; therefore we give 
an ``ad hoc'' definition of regularity for skew-Hamiltonian matrices and justify our choice by proving that the dimension of the 
commutator of regular matrices is indeed minimal. In theory of structured matrices our definition corresponds to that of \emph{non-derogatory} matrices, but 
we prefer to adopt the terminology ``regular'' for consistency with the unstructured and symplectic cases.

\begin{rem}\label{actions}
    The symplectic group $\Sp(n)$ acts on skew-Hamiltonians $\W$ by conjugation. For $M \in \Sp(n)$ and $W \in \W$ one defines:
    \begin{equation}\label{action1}
        M\ast W:= M^{-1} W M,
    \end{equation}
    and this action preserves sums and products.

    By definition, $W \in \W$ if and only if $\tr{W}J=JW$. If $M \in \Sp(n)$ then $\tr{M}J=JM^{-1}$ and $JM=\tr{M^{-1}}J$, 
    and thus if $W$ is skew-Hamiltonian, so is $M^{-1} W M$, because:
\begin{align}
	\nonumber \tr{(M^{-1} W M)}J=\tr{M} &\tr{W} (\tr{M^{-1}} J)= \tr{M} (\tr{W} J) M\\
	\nonumber &= (\tr{M} J) W M=J(M^{-1} W M).
\end{align}
    The fact that this action preserves sums and products is an immediate check.
    
    The symplectic group $\Sp(n)$ also acts on skew-symmetric matrices $\sk$ by congruence. For $M \in \Sp(n)$ and $B \in \sk$:
    \begin{equation}\label{action2}
        M\star B:= \tr{M} B M,
    \end{equation}
    and this action preserves sums and products.
\end{rem} 

\begin{lem}\label{consistent}
	The two actions \eqref{action1} and \eqref{action2} are consistent with each other. 
\end{lem}

\begin{proof} Write a skew-Hamiltonian $W$ as $JB$, with $B$ skew-symmetric. Then, given $M$ symplectic, its inverse $M^{-1}$ is symplectic as well, 
    and we have that:
    \begin{equation}
	M \ast (JB) = M^{-1} \:(JB)\: M = J\:(\tr{M} B M)=J(M\star B).
	\qedhere
\end{equation}
\end{proof}

\begin{lem}\cite[Theorem 3]{waterhouse}\label{up to sp}
	Let $W \in \W$ be a skew-Hamiltonian matrix of even size $n$. Up to symplectic conjugation $W$ is of the form 
    $\begin{bmatrix}
              P & 0\\
              0 & \tr{P}
    \end{bmatrix}$ for some $\frac{n}{2} \times \frac{n}{2}$ matrix $P \in Mat(\frac{n}{2},\C)$.
\end{lem}

\begin{defn}\label{regular}
    Let $W \in \W$ be a skew-Hamiltonian matrix. We call $W$ \emph{regular} if the minimal polynomial of $P$ in Lemma \ref{up to sp} has degree $\frac{n}{2}$. 
We denote the set of regular skew-Hamiltonians by $\W_{reg}$
\end{defn}

If $W$ regular then $P$ is a regular element of the Lie algebra ${\mathfrak{gl}}_{\frac{n}{2}}$. 
In particular for each of its distinct eigenvalues there is, in its Jordan normal form, only one corresponding Jordan block. 
This is equivalent to asking for the minimal polynomial of $P$ to coincide (up to sign) with the characteristic polynomial.

\begin{prop}\label{dimension lemma}
    Let $JB \in \W_{reg}$ be a regular skew-Hamiltonian. Then:
    \begin{equation}\label{dimension formula}
	\mbox{$\big{\{} JA \in \W \:|\: [JA,JB]=0\big{\}} =\langle (JB)^k\:|\:k=0,\ldots,\Big{(}\frac{n}{2}-1\Big{)} \rangle.$}
    \end{equation}
    In particular, the centralizer of $JB$ has minimal dimension $\frac{n}{2}$.
\end{prop}

\begin{proof}
The inclusion $\supseteq$ is immediate. Equality follows for dimensional reasons. Indeed, by Lemma \ref{up to sp} there is a symplectic matrix $M \in \Sp(n)$ such that 
    $M^{-1}(JB)M=\begin{bmatrix}
              P & 0\\
              0 & \tr{P}
    \end{bmatrix}$ for some regular $P$. 
Let us look at all matrices $C=\left[ \begin{array}{cc}
    C_1&C_2\\
    C_3&C_4
    \end{array}
    \right]$ that commute with $JB$. Imposing that:
$$[C,JB]=\left[ \begin{array}{cc}
        [C_1,P]&C_2 \tr{P} -P C_2\\
        C_3 P-\tr{P} C_3&[C_4,\tr{P}]
        \end{array}
        \right]=0$$    
means that $C_2=C_3=0$, while, since $P$ is regular, $C_1$ and $C_4$ are parametrized by $\frac{n}{2}$ degrees of freedom each. 
Imposing to $C$ the extra condition of being skew-Hamiltonian means imposing to 
$JC=\left[ \begin{array}{cc}
0&C_4\\
-C_1&0
\end{array}
\right]$ to be skew-symmetric, and this halves the degrees of freedom to $\frac{n}{2}$. Hence the dimension of the left hand side in \eqref{dimension formula} equals the 
dimension of the right hand side, and they are both equal to $\frac{n}{2}$.
\end{proof}

It is worth recalling part of \cite[Proposition 2.2]{basili}, where it is proved that $(\sk)_{reg}=\sk \cap \{\hbox{regular matrices}\}$ 
is the open subset of all elements having centralizer of minimal dimension $\frac{n}{2}$.

%%%%%%%%%%%%%%%%%%%%%%%%%%%%%%%%%%%%%%%%%%%%%%%%%%%%%%%%%%%%%%%%%%%%%%%%%%%%%%%%%
\subsection{Irreducibility of $\crnb$ and diamond matrices}
%%%%%%%%%%%%%%%%%%%%%%%%%%%%%%%%%%%%%%%%%%%%%%%%%%%%%%%%%%%%%%%%%%%%%%%%%%%%%%%%%

For any skew-symmetric matrix $B \in \sk$, consider the vector space homomorphism:
\begin{equation}\label{phib}
    \xymatrix@R-4ex{\sk \ar[r]^-{\phib}& \sy\\
              A \ar@{|->}[r]&AJB-BJA,}
\end{equation}

and define:

\begin{equation}\label{def cnb}
	\cnb:=\im \phib=\{S \in \sy\:|\:S=AJB-BJA \:\:\hbox{for some}\:A \in \sk\}.
\end{equation}

If $JB$ is regular, since $AJB-BJA=-J[JA,JB]$, then by Lemma \ref{dimension lemma} the defect of the map $\phib$ 
is $\frac{n}{2}$, and $\cnb$ is a linear subspace of symmetric matrices $\sy$ of codimension:
    $$\mbox{${n+1 \choose 2} - \Big{[}{n \choose 2}-\frac{n}{2}\Big{]}=\frac{3}{2}n.$}$$

We can give explicit equations for $\cnb$. We need the following definition, from \cite{note_vanni}:

\begin{defn}\label{superantisub}
Let $M=(m_{ij})$ be a $d \times d$ square matrix, and let $k \in \Z$, $-d < k < d$.
\begin{enumerate}
\item The \emph{$k$-th trace} of $M$ is the sum $\sum_{i=1}^{d-k}m_{i,i+k}$ if $k \ge 0$, or $\sum_{j=1}^{d+k}m_{j-k,j}$ if $k \le 0$. 
The usual trace of a matrix corresponds to the $0$-th trace.\\
If the $k$-th trace is zero for all $k \ge 0$ (respectively $k \le 0$), $M$ is called \emph{supertraceless} (resp. \emph{subtraceless}). 
\item The \emph{$k$-th antitrace} of $M$ is the sum $\sum_{i+j=d+1-k} m_{i,j}$.\\
If the $k$-th antitrace is zero for all $k \ge 0$ (respectively $k \le 0$), $M$ is called \emph{superantitraceless} (resp. \emph{subantitraceless}). 
\end{enumerate}
\end{defn}

\begin{defn}\label{d-block}
	Given any partition $\underline{d}=(d_1,\ldots,d_m)$ of $\frac{n}{2}$, the 
	\emph{$\underline{d}$-block partition} of a $\frac{n}{2} \times \frac{n}{2}$ matrix $X$ is the set of blocks $X_{ij}$, for $i,j=1,\ldots,m$, such that 
	$X_{ij}$ is a $d_i \times d_j$ submatrix of $X$ and
	$$X=\begin{bmatrix}
		X_{11} & \hdots & X_{1m}\\
		\vdots && \vdots\\
		X_{m1} & \hdots & X_{mm}
		\end{bmatrix}$$
\end{defn}

\begin{defn}\label{diamond}
	An $n \times n$ matrix $Y=\begin{bmatrix}
		Y_1 & Y_2\\
		Y_4 & Y_3
		\end{bmatrix} \in V \otimes V$ is called a \emph{diamond matrix} if there exists a partition $\underline{d}$ of $\frac{n}{2}$ such that 
		the diagonal blocks in the $\underline{d}$-block partition of:
		\begin{itemize}
			\item $Y_1$ are superantitraceless, 
			\item the ones of $Y_2$ are supertraceless,
			\item the ones of $Y_3$ are subantitraceless, and 
			\item the ones of $Y_4$ are subtraceless.
		\end{itemize}
The $n$ linear conditions imposing the vanishing of traces and antitraces 
are called the \emph{diamond conditions} ($\di$-conditions), and each of them vanishes on a $\di$-\emph{hyperplane}.
\end{defn}

\vskip.05in

To understand the origin of the terminology ``diamond'', the reader should look at Figure 1 and Figure 2, where diamond matrices corresponding 
respectively to the partition $\underline{d}=(\frac{n}{2})$ and to $\underline{d}=(d_1,d_2,d_3)$ are shown. 
The diagonal lines represent the traces and antitraces that are zero.

\begin{center}

	\scalebox{0.9} % Change this value to rescale the drawing.
	{
	\begin{pspicture}(0,-3.525)(7.025,3.525)
	\psline[linewidth=0.03cm](0.0,0.0)(3.5,3.5)
	\psline[linewidth=0.03cm](0.0,0.5)(3.0,3.5)
	\psline[linewidth=0.03cm](0.0,1.0)(2.5,3.5)
	\psline[linewidth=0.03cm](0.0,1.5)(2.0,3.5)
	\psline[linewidth=0.03cm](0.0,2.0)(1.5,3.5)
	\psline[linewidth=0.03cm](0.0,2.5)(1.0,3.5)
	\psline[linewidth=0.03cm](3.5,-3.5)(7.0,0.0)
	\psline[linewidth=0.03cm](4.0,-3.5)(7.0,-0.5)
	\psline[linewidth=0.03cm](4.5,-3.5)(7.0,-1.0)
	\psline[linewidth=0.03cm](5.0,-3.5)(7.0,-1.5)
	\psline[linewidth=0.03cm](5.5,-3.5)(7.0,-2.0)
	\psline[linewidth=0.03cm](6.0,-3.5)(7.0,-2.5)
	\psline[linewidth=0.03cm](6.5,-3.5)(7.0,-3.0)
	\psline[linewidth=0.03cm](3.5,3.5)(7.0,0.0)
	\psline[linewidth=0.03cm](4.0,3.5)(7.0,0.5)
	\psline[linewidth=0.03cm](4.5,3.5)(7.0,1.0)
	\psline[linewidth=0.03cm](5.0,3.5)(7.0,1.5)
	\psline[linewidth=0.03cm](5.5,3.5)(7.0,2.0)
	\psline[linewidth=0.03cm](6.0,3.5)(7.0,2.5)
	\psline[linewidth=0.03cm](6.5,3.5)(7.0,3.0)
	\psline[linewidth=0.03cm](0.0,0.0)(3.5,-3.5)
	\psline[linewidth=0.03cm](0.0,-0.5)(3.0,-3.5)
	\psline[linewidth=0.03cm](0.0,-1.0)(2.5,-3.5)
	\psline[linewidth=0.03cm](0.0,-1.5)(2.0,-3.5)
	\psline[linewidth=0.03cm](0.0,-2.0)(1.5,-3.5)
	\psline[linewidth=0.03cm](0.0,-2.5)(1.0,-3.5)
	\psline[linewidth=0.03cm](0.0,-3.0)(0.5,-3.5)
	\psline[linewidth=0.03cm](0.0,3.0)(0.5,3.5)
	\psline[linewidth=0.024cm,linestyle=dashed,dash=0.16cm 0.16cm](0.0,0.0)(7.0,0.0)
	\psline[linewidth=0.024cm,linestyle=dashed,dash=0.16cm 0.16cm](3.5,3.5)(3.5,-3.5)
	\psframe[linewidth=0.05,dimen=middle](7.0,3.5)(0.0,-3.5)
	\usefont{T1}{ptm}{m}{n}
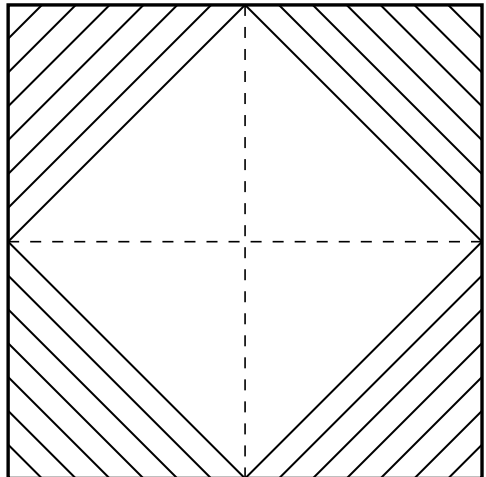
		\rput(3,-3.9){Figure 1. A diamond matrix corresponding to the partition $\underline{d}=(\frac{n}{2})$.}
		\end{pspicture} 
	}

\end{center}

\vskip.3in

\begin{center}
	\scalebox{0.65} % Change this value to rescale the drawing.
{
\begin{pspicture}(0,-5.025)(10.025,5.025)
\psframe[linewidth=0.04,dimen=middle](0.5,5.0)(0.0,4.5)
\psframe[linewidth=0.04,dimen=middle](2.0,4.5)(0.5,3.0)
\psframe[linewidth=0.04,dimen=middle](5.0,3.0)(2.0,0.0)
\psframe[linewidth=0.04,dimen=middle](5.5,5.0)(5.0,4.5)
\psframe[linewidth=0.04,dimen=middle](7.0,4.5)(5.5,3.0)
\psframe[linewidth=0.04,dimen=middle](10.0,3.0)(7.0,0.0)
\psframe[linewidth=0.04,dimen=middle](5.5,0.0)(5.0,-0.5)
\psframe[linewidth=0.04,dimen=middle](7.0,-0.5)(5.5,-2.0)
\psframe[linewidth=0.04,dimen=middle](10.0,-2.0)(7.0,-5.0)
\psframe[linewidth=0.04,dimen=middle](0.5,0.0)(0.0,-0.5)
\psframe[linewidth=0.04,dimen=middle](2.0,-0.5)(0.5,-2.0)
\psframe[linewidth=0.04,dimen=middle](5.0,-2.0)(2.0,-5.0)
\psline[linewidth=0.03cm](0.5,3.0)(2.0,4.5)
\psline[linewidth=0.03cm](0.5,3.5)(1.5,4.5)
\psline[linewidth=0.03cm](0.5,4.0)(1.0,4.5)
\psline[linewidth=0.03cm](0.0,4.5)(0.5,5.0)
\psline[linewidth=0.03cm](2.0,2.5)(2.5,3.0)
\psline[linewidth=0.03cm](2.0,2.0)(3.0,3.0)
\psline[linewidth=0.03cm](2.0,0.0)(5.0,3.0)
\psline[linewidth=0.03cm](2.0,0.5)(4.5,3.0)
\psline[linewidth=0.03cm](2.0,1.0)(4.0,3.0)
\psline[linewidth=0.03cm](2.0,1.5)(3.5,3.0)
\psline[linewidth=0.03cm](5.0,5.0)(5.5,4.5)
\psline[linewidth=0.03cm](5.5,4.5)(7.0,3.0)
\psline[linewidth=0.03cm](7.0,3.0)(10.0,0.0)
\psline[linewidth=0.03cm](6.0,4.5)(7.0,3.5)
\psline[linewidth=0.03cm](6.5,4.5)(7.0,4.0)
\psline[linewidth=0.03cm](7.5,3.0)(10.0,0.5)
\psline[linewidth=0.03cm](8.0,3.0)(10.0,1.0)
\psline[linewidth=0.03cm](8.5,3.0)(10.0,1.5)
\psline[linewidth=0.03cm](9.0,3.0)(10.0,2.0)
\psline[linewidth=0.03cm](9.5,3.0)(10.0,2.5)
\psline[linewidth=0.03cm](0.0,0.0)(0.5,-0.5)
\psline[linewidth=0.03cm](0.5,-0.5)(2.0,-2.0)
\psline[linewidth=0.03cm](2.0,-2.0)(5.0,-5.0)
\psline[linewidth=0.03cm](5.0,-0.5)(5.5,0.0)
\psline[linewidth=0.03cm](5.5,-2.0)(7.0,-0.5)
\psline[linewidth=0.03cm](7.0,-5.0)(10.0,-2.0)
\psline[linewidth=0.03cm](6.0,-2.0)(7.0,-1.0)
\psline[linewidth=0.03cm](6.5,-2.0)(7.0,-1.5)
\psline[linewidth=0.03cm](7.5,-5.0)(10.0,-2.5)
\psline[linewidth=0.03cm](8.0,-5.0)(10.0,-3.0)
\psline[linewidth=0.03cm](8.5,-5.0)(10.0,-3.5)
\psline[linewidth=0.03cm](9.0,-5.0)(10.0,-4.0)
\psline[linewidth=0.03cm](9.5,-5.0)(10.0,-4.5)
\psline[linewidth=0.03cm](0.5,-1.0)(1.5,-2.0)
\psline[linewidth=0.03cm](0.5,-1.5)(1.0,-2.0)
\psline[linewidth=0.03cm](2.0,-2.5)(4.5,-5.0)
\psline[linewidth=0.03cm](2.0,-3.0)(4.0,-5.0)
\psline[linewidth=0.03cm](2.0,-3.5)(3.5,-5.0)
\psline[linewidth=0.03cm](2.0,-4.0)(3.0,-5.0)
\psline[linewidth=0.03cm](2.0,-4.5)(2.5,-5.0)
\psline[linewidth=0.024cm,linestyle=dashed,dash=0.16cm 0.16cm](5.0,5.0)(5.0,-5.0)
\psline[linewidth=0.024cm,linestyle=dashed,dash=0.16cm 0.16cm](0.0,0.0)(10.0,0.0)
\psframe[linewidth=0.05,dimen=middle](10.0,5.0)(0.0,-5.0)
\usefont{T1}{ptm}{m}{n}
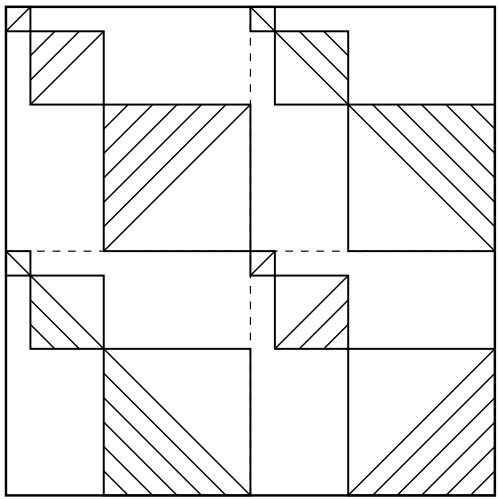
\rput(5,-5.5){\Large{Figure 2. A diamond matrix corresponding to a partition $\underline{d}=(d_1,d_2,d_3)$.}}
\end{pspicture} 
}
\end{center}

\vskip.3in

\begin{prop}\cite[Corollary 3.13]{note_vanni}\label{AJB-BJA diamond} 
Let $B \in \sk$ be such that $JB \in \W_{reg}$ is a regular skew-Hamiltonian, and let $A \in \sk$ be any skew-symmetric matrix. 
Then $AJB-BJA$ is symplectically congruent to a diamond matrix. 
\end{prop}

\begin{proof} 
	We start by making the observation that the morphism $\phib$ is $\Sp(n)$-equivariant under the congruence action. (It is easy to check that the group $\Sp(n)$ acts by congruence not only on skew-symmetric, but also on symmetric matrices.) 

Given any skew-symmetric matrix $A$ one has:
	\begin{equation}\label{equivariancy}
		M \star \phib(A)=\tr{M} \: (\phib(A)) \: M = \varphi^{\tr{M}B M}(\tr{M}A M)=\varphi^{M \star B}(M \star A).
	\end{equation}
	
	Now put $JB$ in its Jordan normal form via symplectic conjugation. The Jordan normal form of $JB$ will consist in $m$ Jordan blocks, each of dimension $d_i$ with $\sum_{i=1}^m d_i=\frac{n}{2}$. Set $\underline{d}=(d_1,\ldots,d_m)$ as before. For any $A \in \sk$, 
consider the $\underline{d}$-block partitions of the four quadrants of $S=\phib(A)$. A direct computation now shows that $S$ is a diamond matrix.
\end{proof}

\begin{rem}
	Notice that for a symmetric matrix the $\di$-conditions reduce to $\frac{3}{2}n$ conditions, 
	simply because if the blocks of $S_2$ are supertraceless, then the ones of $\tr{S_2}$ will automatically be subtraceless.
\end{rem}

\vskip.1in

Denote by $\syrc$ the determinantal variety:
$$\syrc:=\{ S \in S^2V \:|\: \rk S \le r\},$$
which is irreducible of dimension $nr-{r \choose 2}$. 
There are several proofs of this fact, we briefly recall here the one due to Kempf \cite{kempf}. 

Let $\G(n-r,n)$ be the Grassmannian of $n-r$-dimensional subspaces of the $n$-dimensional vector space $V$, 
and consider the incidence variety:
\begin{equation}\label{affine bundle}
	\tilde{X}:=\{(L,S) \in \G(n-r,n)\times \sy\:|\:L \subseteq \Ker S\},
\end{equation}
that comes with the two standard projections:
\begin{equation}
	\xymatrix@C-4ex{&\tilde{X}\ar[dl]_{\pi_1}\ar[dr]^{\pi_2}&\\
	\G(n-r,n)&&\:\:\:\:\:\sy\:\:\:\:\:
	}
\end{equation}

It is an affine bundle over the Grassmannian, whose fibers are vector spaces of dimension equal to $\dim S^2(V/L)={r+1 \choose 2}$. 
The canonical projection of $\tilde{X}$ on $\sy$ is surjective and the 
restriction of it to the inverse image of the open subset of all symmetric matrices of rank $r$ is injective; 
hence $\syrc$ is irreducible of dimension: 
$$\mbox{$\dim \syrc= \dim \tilde{X}= \dim \G(n-r,n)+ {r+1 \choose 2}=nr - {r \choose 2}.$}$$

\vskip.2in

Intersecting the linear space $\cnb$ with the determinantal variety $\syrc$ we define:
\begin{equation}\label{def crnb}
	\crnb:=\cnb \cap \syrc = \{S \in \sy\:|\:S \in \cnb,\: \rk S \le r\}.
\end{equation}

\begin{lem}\label{C_b irred}
	If $B \in \sk$ is such that $JB \in \W_{reg}$ is regular, then $\crnb$ is irreducible of dimension $nr-\frac{3}{2}n-{r \choose 2}$ for $r=n$ and $n \ge 4$, and 
	for $r=n-1$ and $n \ge 8$.
\end{lem}

\begin{proof}
The case $r=n$ is trivial, hence in what follows we will concentrate on the case $r=n-1$.

Similarly to what is done above, define the incidence variety :
$$\tilde{X}_r=\{(L,S)\in\G(n-r,n) \times \syrc\:|\: L \subseteq \Ker S\}.$$ 

Again, we have the two standard projections:
\begin{equation}\label{incidence}
	\xymatrix@C-3ex{&\tilde{X}_r\ar[dl]_{p_1}\ar[dr]^{p_2}&&\\
	\G(n-r,n)&&\syrc \ar@{<-^{)}}[r]&\crnb
	}
\end{equation}

In order to prove Lemma \ref{C_b irred}, it is enough to show that $p_2^{-1}(\crnb)$ is irreducible. We will do so for $r=n-1$. 

By definition, the variety $\crnb$ is the intersection of $\syrc$ with $\cnb$, which is a vector space of codimension $\frac{3}{2}n$. Let $K \subset S^2 V$ be a general vector space of codimension $\frac{3}{2} n$. Since $\frac{3}{2} n < nr - {r \choose 2}$, we can apply Bertini's theorem for a proper morphism, 
see \cite[Theorem 3.3.1]{Positivity}, and we obtain that $p_2^{-1}(K)$ is irreducible of dimension $nr - {r \choose 2} - \frac{3}{2} n$. 
Then, using \cite[Expos\'e $\mathrm{XIII}$, coroll. $2.2$]{SGA2}, we deduce that $p_2^{-1}(\cnb) = p_2^{-1}(\crnb)$ is connected in dimension $nr - {r \choose 2} - \frac{3}{2} n -1$.

This connectedness result allow us to reduce the proof of the irreducibility of $p_2^{-1}(\crnb)$ to showing that it has the expected dimension $\dim S^2V_{\leq r} - \frac{3}{2}n$, 
and that it is smooth in codimension 1, that is, its singular locus is in codimension at least $2$.

To prove that this is indeed true, we study the projection ${p_1}|_{p_2^{-1}(\crnb)}$:
$$
{p_1}|_{p_2^{-1}(\crnb)} : p_2^{-1}(\crnb) \rightarrow \G(n-r,n).
$$

Recall that $p_2^{-1}(\crnb)$ is defined as:
$$
p_2^{-1}(\crnb)=\{(L,S)\in \tilde{X}_r\:|\: S\:\hbox{satisfies the}\:\di\hbox{-conditions}\}.
$$

Hence given an element $L$ of the Grassmannian $\G(n-r,n)$, the fiber of ${p_1}|_{p_2^{-1}(\crnb)}$ over $L$ is:
$$
{p_1}^{-1}|_{p_2^{-1}(\cnb)}(L) = \{S \in S^2V_{\leq r} \:|\: L \subseteq \Ker S,\: S\:\hbox{satisfies the}\:\di\hbox{-conditions}\}.
$$

We wish to identify the elements $L$ in the Grassmannian whose fiber ${p_1}^{-1}|_{p_2^{-1}(\crnb)}(L)$ is not a linear space of the expected dimension 
${r+1 \choose 2} - \frac{3}{2} n$. In other words, we are looking for all $L$ in $\G(n-r,n)$ for which the conditions 
``$L \subseteq \Ker S$" are not independent from the $\di$-conditions.

\vskip.05in

Let us denote by $\{e_1,\ldots,e_n\}$ the basis of $V$ with respect to which the $\di$-conditions are represented. We recall that it is 
the basis in which the regular skew-Hamiltonian $JB$ is in its Jordan normal form.  

\underline{From now on, we set $r=n-1$.} Let $L \in \G(1,n)$ be generated by $L=\langle\: \ell \:\rangle$, with 
$\ell=\sum_{k=1}^n \lambda_k e_k$. Then if $S=(s_{pq})$, the condition $L \subseteq \Ker S$ translates into the $n$ conditions, one for 
each $j=1,\ldots,n$, that we denote by:
$$
\sh{}^j:\:\:\:\:\:\:\:\:\:\:\:\:\:
\sum_{k=1}^n \lambda_k s_{kj} =0$$

Notice that each $\sh{}^j$ involves the $n$ entries of the $j$-th column of the matrix $S$. For this reason, we will think of them as \emph{vertical} conditions. 

The symmetry of the matrix $S$ yields that for any vertical condition $\sh{}^j$ there is a \emph{horizontal} one, involving the $j$th row of $S$. 
We denote such horizontal condition by $\sh{}_j$, with obvious notation.

In the same mindset, the $\di$-conditions, involving traces and anti-traces, are thought of as \emph{diagonal} conditions. 
The number of entries of $S$ that are involved in a diagonal condition varies from 1 to $n/2$.

The notation has been chosen in order to remind the reader whether a particular condition is diagonal ($\di$), or vertical and horizontal ($\sh$). 

\vskip.05in

Consider now a $\di$-condition $d$. Notice that it involves entries of the matrix $S$ that are all in one of the four quadrants. 
We say that $d$ \emph{is generated} by $\sh$-conditions if there exist 
some vertical and horizontal conditions $v_i$ and $h_j$ such that it is possible to write $d= \sum \alpha_iv_i + \sum \beta_jh_j + \delta$, 
where $\delta$ is a linear form that only involves entries of $S$ that are not in the same quadrant of $d$.

\begin{center}
\scalebox{0.65} % Change this value to rescale the drawing.
{
\begin{pspicture}(0,-5.02)(10.04,5.02)
\definecolor{fucsia}{rgb}{0.9,0.2,1.0}
\definecolor{blu}{rgb}{0.2,0.4,1.0}
\definecolor{rosso}{rgb}{1.0,0.2,0.0}
\psframe[linewidth=0.04,dimen=middle](0.5,5.02)(0.0,4.52)
\psframe[linewidth=0.04,dimen=middle](2.0,4.52)(0.5,3.02)
\psframe[linewidth=0.04,dimen=middle](5.0,3.02)(2.0,0.02)
\psframe[linewidth=0.04,dimen=middle](5.5,5.02)(5.0,4.52)
\psframe[linewidth=0.04,dimen=middle](7.0,4.52)(5.5,3.02)
\psframe[linewidth=0.04,dimen=middle](10.0,3.02)(7.0,0.02)
\psframe[linewidth=0.04,dimen=middle](5.5,0.02)(5.0,-0.48)
\psframe[linewidth=0.04,dimen=middle](7.0,-0.48)(5.5,-1.98)
\psframe[linewidth=0.04,dimen=middle](10.0,-1.98)(7.0,-4.98)
\psframe[linewidth=0.04,dimen=middle](0.5,0.02)(0.0,-0.48)
\psframe[linewidth=0.04,dimen=middle](2.0,-0.48)(0.5,-1.98)
\psframe[linewidth=0.04,dimen=middle](5.0,-1.98)(2.0,-4.98)
\psline[linewidth=0.03cm](0.52,3.0)(2.02,4.5)
\psline[linewidth=0.03cm](0.52,3.5)(1.52,4.5)
\psline[linewidth=0.03cm](0.52,4.0)(1.02,4.5)
\psline[linewidth=0.03cm](0.02,4.5)(0.52,5.0)
\psline[linewidth=0.03cm](2.02,2.5)(2.52,3.0)
\psline[linewidth=0.03cm](2.02,2.0)(3.02,3.0)
\psline[linewidth=0.03cm](2.02,0.0)(5.02,3.0)
\psline[linewidth=0.03cm](2.02,0.5)(4.52,3.0)
\psline[linewidth=0.07cm,linecolor=fucsia](2.02,1.0)(4.02,3.0)
\psline[linewidth=0.03cm](2.02,1.5)(3.52,3.0)
\psline[linewidth=0.03cm](5.02,5.0)(5.52,4.5)
\psline[linewidth=0.03cm](5.52,4.5)(7.02,3.0)
\psline[linewidth=0.03cm](7.02,3.0)(10.02,0.0)
\psline[linewidth=0.03cm](6.02,4.5)(7.02,3.5)
\psline[linewidth=0.03cm](6.52,4.5)(7.02,4.0)
\psline[linewidth=0.03cm](7.52,3.0)(10.02,0.5)
\psline[linewidth=0.03cm](8.02,3.0)(10.02,1.0)
\psline[linewidth=0.03cm](8.52,3.0)(10.02,1.5)
\psline[linewidth=0.03cm](9.02,3.0)(10.02,2.0)
\psline[linewidth=0.03cm](9.52,3.0)(10.02,2.5)
\psline[linewidth=0.03cm](0.02,0.0)(0.52,-0.5)
\psline[linewidth=0.03cm](0.52,-0.5)(2.02,-2.0)
\psline[linewidth=0.03cm](2.02,-2.0)(5.02,-5.0)
\psline[linewidth=0.03cm](5.02,-0.5)(5.52,0.0)
\psline[linewidth=0.03cm](5.52,-2.0)(7.02,-0.5)
\psline[linewidth=0.03cm](7.02,-5.0)(10.02,-2.0)
\psline[linewidth=0.03cm](6.02,-2.0)(7.02,-1.0)
\psline[linewidth=0.03cm](6.52,-2.0)(7.02,-1.5)
\psline[linewidth=0.03cm](7.52,-5.0)(10.02,-2.5)
\psline[linewidth=0.03cm](8.02,-5.0)(10.02,-3.0)
\psline[linewidth=0.03cm](8.52,-5.0)(10.02,-3.5)
\psline[linewidth=0.03cm](9.02,-5.0)(10.02,-4.0)
\psline[linewidth=0.03cm](9.52,-5.0)(10.02,-4.5)
\psline[linewidth=0.03cm](0.52,-1.0)(1.52,-2.0)
\psline[linewidth=0.03cm](0.52,-1.5)(1.02,-2.0)
\psline[linewidth=0.03cm](2.02,-2.5)(4.52,-5.0)
\psline[linewidth=0.03cm](2.02,-3.0)(4.02,-5.0)
\psline[linewidth=0.03cm](2.02,-3.5)(3.52,-5.0)
\psline[linewidth=0.03cm](2.02,-4.0)(3.02,-5.0)
\psline[linewidth=0.03cm](2.02,-4.5)(2.52,-5.0)
\psline[linewidth=0.024cm,linestyle=dashed,dash=0.16cm 0.16cm](5.02,5.0)(5.02,-5.0)
\psline[linewidth=0.024cm,linestyle=dashed,dash=0.16cm 0.16cm](0.02,0.0)(10.02,0.0)
\psframe[linewidth=0.05,dimen=middle](10.0,5.02)(0.0,-4.98)
\psline[linewidth=0.03cm,linecolor=blu](2.27,4.97)(2.27,-4.93)
\psline[linewidth=0.03cm,linecolor=blu](2.77,4.97)(2.77,-4.93)
\psline[linewidth=0.03cm,linecolor=blu](3.27,4.97)(3.27,-4.93)
\psline[linewidth=0.03cm,linecolor=blu](3.77,4.97)(3.77,-4.93)
\psline[linewidth=0.03cm,linecolor=rosso](0.05,2.75)(9.95,2.75)
\psline[linewidth=0.03cm,linecolor=rosso](0.05,2.25)(9.95,2.25)
\psline[linewidth=0.03cm,linecolor=rosso](0.05,1.75)(9.95,1.75)
\psline[linewidth=0.03cm,linecolor=rosso](0.05,1.25)(9.95,1.25)
\psdots[dotsize=0.24200001,dotstyle=asterisk](2.27,1.25)
\psdots[dotsize=0.24200001,dotstyle=asterisk](2.77,1.75)
\psdots[dotsize=0.24200001,dotstyle=asterisk](3.27,2.25)
\psdots[dotsize=0.24200001,dotstyle=asterisk](3.77,2.75)
\usefont{T1}{ptm}{m}{n}
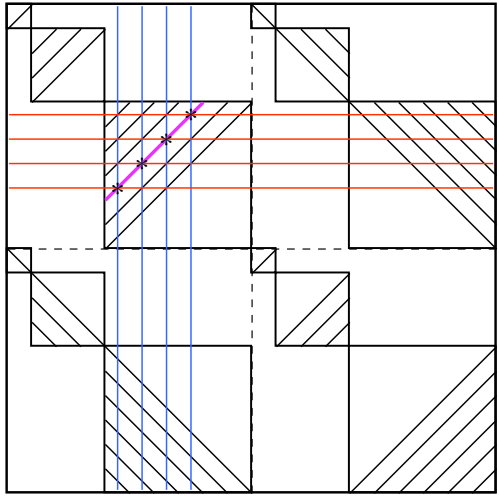
\rput(5,-5.5){\Large{Figure 3. Representation of a diagonal condition generated by vertical and horizontal ones.}}
\end{pspicture} 
}
\end{center}

\vskip.3in

\begin{lem}\label{diag=vert+horiz}
To generate a diagonal $\di$-condition involving $z$ entries, one needs $x$ vertical conditions of type $\sh{}^i$ and 
$y$ horizontal conditions of type $\sh_{}j$, with $z=x+y$.
\end{lem}

\begin{proof}
In light of the remarks above, the proof is almost immediate.

We fix any integer $1 \le z \le \frac{n}{2}$. A $\di$-condition involving $z$ entries of $S$ is either a trace or an antitrace in one of the $m$ blocks in which $S$ is 
block-partitioned. More in detail, there is an index $1 \le \iota \le m$ such that this $\di$-condition is either 
a $\pm(d_\iota-z)$-th trace or a $\pm(d_\iota-z)$-th antitrace. We only look at the case when it is a $(d_\iota-z)$-th antitrace, since the proof goes through verbatim in the other cases. 
In order for this antitrace condition to be generated by the $\sh$-conditions, we need to see where do its $z$ entries appear. 
Recall from Definition \ref{superantisub}(2) that they are entries of type $s_{pq}$, with $p+q=z+1$. 

Each of these $z$ entries $s_{pq}$ can show up in at most one vertical condition, namely $\sh{}^q$, and 
at most one horizontal condition, namely $\sh{}_p$. And in each of these the entry $s_{pq}$ appears exactly once, that is 
for $k=z+1-q$ and $k=z+1-p$ respectively. Hence the statement follows.
\end{proof}

We now make the following:
\vskip.05in
\emph{Claim 1.} Let the Jordan normal form of $JB$ consist in $m$ Jordan blocks, each of dimension $d_i$ with $\sum_{i=1}^m d_i=n/2$. 
For every $i=1,\ldots,m$ define $\delta_i:=\sum_{j=1}^{i-1} d_j$. Let $L=\langle\: \ell \:\rangle$ be an element in the Grassmannian $\G(1,n)$.
Then the condition $L \subseteq \Ker S$ is non-transverse with the $\di$-conditions if and only if: 
\begin{equation}\label{intersection}
	\mbox{$L \subset \bigcup_{i=1}^m \langle e_{\delta_i+1},e_{\delta_{i+1}+\frac{n}{2}}\rangle.$}
\end{equation}
\vskip.1in 

\emph{Proof of claim 1.} The implication $(\Leftarrow)$ is immediate. If $L$ is included in $\bigcup_{i=1}^m \langle e_{\delta_i+1},e_{\delta_{i+1}+\frac{n}{2}}\rangle$ , then 
there is an index $1 \le \alpha \le m$ such that the generator $\ell \in \langle e_{\delta_\alpha +1},e_{\delta_{\alpha +1}+\frac{n}{2}}\rangle$. This means that 
the condition $\sh{}^{\delta_\alpha + 1}$ reads:
$$\lambda_{\delta_\alpha + 1} s_{\delta_\alpha + 1,\delta_\alpha + 1} + \lambda_{\delta_{\alpha +1}+\frac{n}{2}} s_{\delta_{\alpha +1}+\frac{n}{2},\delta_\alpha + 1}=0.$$

Divide as usual $S=\begin{bmatrix}
          S_1 & S_2\\
          \tr{S_2} & S_3
\end{bmatrix}$
in four quadrants. Among the $\di$-conditions, the vanishing of the $d_\alpha -1$th antitrace of the block $(S_1)_{\alpha\alpha}$ entails that the entry $s_{\delta_\alpha + 1,\delta_\alpha + 1}=0$, while the vanishing of the $-d_\alpha +1$th trace of the block $(\tr{S_2})_{\alpha\alpha}$ entails that the entry $s_{\delta_{\alpha +1}+\frac{n}{2},\delta_\alpha + 1}=0$. Hence $\sh{}^{\delta_\alpha +1}$ is trivially satisfied, and the $\di$-conditions are not independent from the conditions $L \subset \Ker S$.

\vskip.05in

To prove $(\Rightarrow)$, we show that if $L$ is not included in $\bigcup_{i=1}^m \langle e_{\delta_i+1},e_{\delta_{i+1}+\frac{n}{2}}\rangle$, then none of the conditions generating $L \subseteq \Ker S$ 
can be obtained from a combination of $\di$-conditions. Indeed, assume that $L$ is not included in $\bigcup_{i=1}^m \langle e_{\delta_i+1},e_{\delta_{i+1}+\frac{n}{2}}\rangle$, then: 
there exist indices $1 \le \beta < \gamma \le n$ such that:
\begin{enumerate}
\item either $\beta$ is different from all $\delta_i+1$ for $i=1,\ldots,m$, and the coefficient $\lambda_\beta \neq 0$,
\item or if $\beta=\delta_\alpha+1$ for an index $1 \le \alpha \le m$, then $\gamma \neq \delta_{\alpha+1}+\frac{n}{2}$, 
and both coefficients $\lambda_\beta$ and $\lambda_\gamma$ are non-zero. 
\end{enumerate}

In both instances, in every $\sh{}^j$ (and in every $\sh{}_j$) the entries $s_{\beta j}$ (and respectively $s_{j \beta}$) 
will show up with non-zero coefficients. Recall that these same entries do not show up in any other $\sh$-conditions. 

Moreover in both instances the $\beta$-row of the matrix $S$ will cut one of the blocks in which $S$ is partitioned, say the $\iota$-th block. 
Then as soon as $j$ is either smaller than $\delta_\iota$, or bigger than $\delta_{\iota +1}$, none of the indices $s_{\beta j}$ will appear in 
a $\di$-condition, and this concludes the proof.\qed

\vskip.1in

Claim 1 can be re-interpreted as saying that transversality fails whenever there is a whole line contained in the intersection \eqref{intersection}. We call these lines \emph{bad lines}. 
Notice that the bad lines form a subvariety of dimension $1$. 

\vskip.15in

\emph{Claim 2.} Let $L$ be an element of the Grassmannian $\G(1,n)$. Then the dimension of the fiber ${p_1}^{-1}|_{p_2^{-1}({\mathcal{S}}^B_{n-1,n})}(L)$ is bounded above by: 
$$\mbox{$\dim ({p_1}^{-1}|_{p_2^{-1}({\mathcal{S}}^B_{n-1,n})}(L) )\le {n \choose 2} - (\frac{3}{2}n - 3),$}$$
where ${n \choose 2}=\dim S^2(V/L)$ is the dimension of the fiber of the affine bundle $\tilde{X} \to \G(1,n)$ defined in \eqref{affine bundle}.

\vskip.1in

\emph{Proof of claim 2.}
To prove this inequality, we have to evaluate the number of $\di$-conditions which can be generated by the condition $ L \subset \Ker S$. 
A close look at Figure $3$ should convince the reader that either two or three among the $\di$-conditions are generated by the condition $L \subset \Ker S$ if and only if:
\begin{equation}
	\mbox{$L \subset \bigcup_{i=1}^m \langle e_{\delta_i+1},e_{\delta_{i+1}+\frac{n}{2}}\rangle.$}
\end{equation} 
Indeed, if $L$ is not included in $\bigcup_{i=1}^m \langle e_{\delta_i+1},e_{\delta_{i+1}+\frac{n}{2}}\rangle$, 
then using claim $1$ we deduce that the $\di$-conditions and the $L \subset \Ker S$ are independent. 

Conversely, assume that 
$L$ is generated by $\mu e_{\delta_i+1} + \lambda e_{\delta_{i+1}+\frac{n}{2}}$ for some $1 \leq i \leq m$ and $\mu, \lambda \in \mathbb{C}$ non simultaneously zero. 

Suppose that $d_i = 1$. Then the linear space $M_{L \subset \Ker S}$ generated by the conditions:
$$\left\{\begin{array}{l}
\mu s_{1,\delta_i+1} + \lambda s_{1,\delta_{i+1}+\frac{n}{2}} = 0\\
\vdots\\
\mu s_{n,\delta_i+1} + \lambda s_{n,\delta_{i+1}+\frac{n}{2}} = 0
\end{array}\right.$$
intersects the linear space generated by the conditions:
$$\left\{\begin{array}{l}
 s_{\delta_i+1,\delta_i+1} = 0\\
s_{\delta_i+1,\delta_{i+1}+\frac{n}{2}}=0\\
s_{\delta_{i+1}+ \frac{n}{2},\delta_{i+1}+\frac{n}{2}} = 0
\end{array}\right.\:\:\:\:\:\:\:\:\:\:$$
in dimension $2$, and no other $\di$-conditions are included in $M_{L \subset \Ker S}$. 

If instead $d_i > 1$, then that same linear space $M_{L \subset \Ker S}$ intersects the linear space generated by:
$$\left\{\begin{array}{l}
s_{\delta_i+1,\delta_i+1} =0 \\
s_{\delta_i+2,\delta_i+1}+s_{\delta_i+1,\delta_i+2} =0 \\
s_{\delta_i+1,\delta_{i+1}+\frac{n}{2}}=0\\
s_{\delta_i+1,\delta_{i+1}+\frac{n}{2}-1} + s_{\delta_i+2,\delta_{i+1}+\frac{n}{2}}=0\\ 
s_{\delta_{i+1}+ \frac{n}{2},\delta_{i+1}+\frac{n}{2}} = 0\\
s_{\delta_{i+1}+ \frac{n}{2},\delta_{i+1}+\frac{n}{2}-1} + s_{\delta_{i+1}+ \frac{n}{2}-1,\delta_{i+1}+\frac{n}{2}} = 0
\end{array}\right.$$
in dimension $3$, and no other $\di$-conditions are included in $M_{L \subset \Ker S}$. This concludes the proof of claim $2$.
\qed

\vskip.15in

The two claims allow us to give an estimate of  the dimension of the locus $\mathcal{B}_{n-1,n} \subset p_2^{-1}({\mathcal{S}}^B_{n-1,n})$ where the intersection of the $\di$-conditions with 
the $\sh$-conditions is non-transverse. Such dimension is bounded above 
by the sum of the dimension of the variety of bad lines and the dimension of the fiber. So altogether:

\begin{equation}\label{mauvais lieu}
	\mbox{$\dim \mathcal{B}_{n-1,n} \le 1 + {n \choose 2} - (\frac{3}{2}n -3).$}
\end{equation}

On the other hand:

\begin{equation}\label{dim Ytilde}
	\mbox{$\dim p_2^{-1}({\mathcal{S}}^B_{n-1,n})= \dim S^2V_{\le n-1} -\frac{3}{2}n=n(n-1)-{n-1 \choose 2}-\frac{3}{2}n$},
\end{equation}

and combining \eqref{mauvais lieu} and \eqref{dim Ytilde} we get:

$$\mbox{$\codim_{p_2^{-1}({\mathcal{S}}^B_{n-1,n})}(\mathcal{B}_{n-1,n})\ge n-5.$}$$

We can thus conclude that the codimension of $\mathcal{B}_{n-1,n}$ in $p_2^{-1}({\mathcal{S}}^B_{n-1,n})$ 
is strictly greater than 2 as soon as $n \ge 8$. In this range we obtain the irreducibility of $p_2^{-1}({\mathcal{S}}^B_{n-1,n})$, and hence that of ${\mathcal{S}}^B_{n-1,n}$, which is what we wanted. 
\end{proof}

%%%%%%%%%%%%%%%%%%%%%%%%%%%%%%%%%%%%%%%%%%%%%%%%%%%%%%%%%%%%%%%%%%%%%%%%%%%%%%%%%
\subsection{Proof of the key lemma}
%%%%%%%%%%%%%%%%%%%%%%%%%%%%%%%%%%%%%%%%%%%%%%%%%%%%%%%%%%%%%%%%%%%%%%%%%%%%%%%%%

We need the following well-known irreducibility criterion. 

\begin{lem}\label{criterion}
	Let $f: X \to Y$ be a surjective morphism of algebraic varieties. Suppose that $Y$ is irreducible, 
	and that all the fibers of $f$ are irreducible of the same dimension $d$. 
	If $X$ is pure-dimensional, then it is irreducible of dimension $\dim Y +d$.
\end{lem}

\begin{lem}\label{open irred}
    Let V be a complex vector space of even dimension $n$. The subvariety:
    $$\crn^0:=\{(A,B) \in \sk \times \sk\:|\: \rk (AJB-BJA)\le r \:\hbox{and}\:JB\:\hbox{is regular}\}$$
    is irreducible of dimension $nr-n + {n \choose 2} - {r \choose 2}$, for $r=n \ge 4$ and for $r=n-1$, and $n \ge 8$.
\end{lem}

\begin{proof}
    Let $\W_{reg}$ be the open subset of all regular skew-Hamiltonians. We have a map:
    \begin{equation}
	    \xymatrix@R-4ex{\crn^0 \ar[r]^-{p}& \W_{reg}\\
	              (A,B) \ar@{|->}[r]&JB}
	\end{equation}

   By the affine Bezout theorem, $\crnb$ is non-empty, and $p$ is surjective. 
   Consider now a fiber $F^B=p^{-1}(B)=\{A \in \sk\:|\:\rk[JA,JB]\le r\}$. There is an epimorphism:
   \begin{equation}
	    \xymatrix@R-4ex{F^B \ar[r]^-{q}& \crnb\\
	              A \ar@{|->}[r]&-J[JA,JB]}
	\end{equation}
	
    Since $JB$ is regular, by Lemma \ref{dimension lemma} $F^B$ is an affine $\C^{n/2}$-bundle over $\crnb$, and hence it is irreducible of dimension:
    $$\mbox{$\dim F^B= \dim \crnb + \frac{n}{2}= nr-n-{r \choose 2}.$}$$ 
    Thus for every irreducible component $C$ of $\crn^0$: 
    $$\mbox{$\dim C \le {n \choose 2}+ \dim F^B= {n \choose 2}+ nr-n  - {r \choose 2}$}.$$ 

Counting equations one sees that the dimension of such components must be at least:
$$\mbox{$2{n \choose 2}-\Big{[}{n+1 \choose 2}-\Big{(}nr -{r \choose 2}\Big{)}\Big{]}= {n \choose 2} + nr-n - {r \choose 2}.$}$$

All in all $\crn^0$ fibers over the irreducible variety $\W_{reg}$, all the fibers are irreducible of the same dimension, and $\crn^0$ is pure-dimensional. 
Hence Lemma \ref{criterion} applies, and the statement follows.
\end{proof}

\vskip.2in

\emph{Proof of Lemma \ref{key lemma}.} 
We are left to prove that $\crn=\overline{\crn^0}$. Take a pair $(A,B)$ with $\rk[JA,JB]=r$ where $JB$ is not regular. 
It is possible to find a regular skew-Hamiltonian $JB'$ such that $[JA,JB']=0$. Then $JB+ t\:JB'$ will be regular for all but finitely many $t$ and 
$[JA,JB+ t\:JB']=[JA,JB]$. This concludes the proof. \qed

%%%%%%%%%%%%%%%%%%%%%%%%%%%%%%%%%%%%%%%%%%%%%%%%%%%%%%%%%%%%%%%%%%%%%%%%%%%%%%%%%
%%%%%%%%%%%%%%%%%%%%%%%%%%%%%%%%%%%%%%%%%%%%%%%%%%%%%%%%%%%%%%%%%%%%%%%%%%%%%%%%%
\section{Dominant maps and low-dimensional cases}
%%%%%%%%%%%%%%%%%%%%%%%%%%%%%%%%%%%%%%%%%%%%%%%%%%%%%%%%%%%%%%%%%%%%%%%%%%%%%%%%%
%%%%%%%%%%%%%%%%%%%%%%%%%%%%%%%%%%%%%%%%%%%%%%%%%%%%%%%%%%%%%%%%%%%%%%%%%%%%%%%%%

% %%%%%%%%%%%%%%%%%%%%%%%%%%%%%%%%%%%%%%%%%%%%%%%%%%%%%%%%%%%%%%%%%%%%%%%%%%%%%%%%%
% \subsection{A curious result}
% %%%%%%%%%%%%%%%%%%%%%%%%%%%%%%%%%%%%%%%%%%%%%%%%%%%%%%%%%%%%%%%%%%%%%%%%%%%%%%%%%

The study of the varieties $\crn$ has led us to prove a somewhat curious result, that we have not been able to find 
anywhere in literature, and that we therefore report here. Similar results were recently obtained by \cite{note_vanni} using completely 
different techniques. 

\begin{prop}\label{phi dominante}
Let $V$ be a complex vector space of even dimension $n$. Define $\phi$ as:
$$\xymatrix@R-4ex{\PP(\sk \times \sk) \ar@{-->}[r]^{\phi}&\:\:\:\:\:\:\:\:\PP(\sy)\:\:\:\:\:\:\:\:\\
\:\:\:\:([A],[B]) \:\:\:\:\ar@{|->}[r] & [AJB-BJA].}$$
    \begin{itemize} 
        \item[a)] For $n=2$, $\phi$ is not defined. 
        \item[b)] For $n=4$, $\im \phi$ is the Grassmannian $\G(2,5)$ in $\PP^9=\PP(S^2 \C^4)$. 
        \item[c)] For $n=6$, $\im \phi$ is a hypersurface of degree 4 in $\PP^{20}=\PP(S^2 \C^6)$.
        \item[d)] For $n \ge 8$, $\phi$ is locally of maximal rank. In particular, $\phi$ is dominant.
    \end{itemize}
\end{prop}

\begin{proof}
\begin{itemize} 
\item[a)] By direct computation, $AJB-BJA=0$ for all $A,B \in \sk$.

\item[b)] We look at the bilinear map $\phi$ as a tensor map: 
$$\PP(\Lambda^2 \C^4) \times \PP(\Lambda^2 \C^4) \hookrightarrow \PP(\Lambda^2 \C^4 \otimes \Lambda^2 \C^4) \dashrightarrow \PP (S^2\C^4).$$ 

Being skew-symmetric, this map factorizes through:
$$\PP(\Lambda^2 (\Lambda^2 \C^4)) \dashrightarrow^{\Phi} \PP (S^2 \C^4).$$

The image of $\phi$ is now the image through $\Phi$ of the Grassmannian of $2$-planes in $\C^6 \simeq \Lambda^2 \C^4$: 
$\G(2,6) \hookrightarrow \PP(\Lambda^2 (\Lambda^2 \C^4))=\PP(\Lambda^2\C^6)$. 

The same argument that we used in the proof of Proposition \ref{AJB-BJA diamond} shows that map is $\Sp(4)$-equivariant. 
We can thus look at the decomposition of $\Lambda^2 (\Lambda^2 \C^4)$ and $S^2 \C^4$ 
as $\Sp(4)$-modules. We see that $S^2 \C^4$ is irreducible, while $\Lambda^2 \C^4$ decomposes as $\Lambda_0^2 \C^4 \oplus \C \simeq \C^5 \oplus \C$, so that:

\begin{equation}\label{two summands}
    \Lambda^2(\Lambda^2 \C^4)= \Lambda^2(\Lambda_0^2 \C^4 \oplus \C)=\Lambda^2(\Lambda_0^2 \C^4) \oplus \Lambda_0^2 \C^4 \simeq \C^{10} \oplus \C^5.
\end{equation}

Schur's lemma now tells us that $\Phi$ maps $\PP(\Lambda^2(\Lambda_0^2 \C^4))$ isomorphically on $\PP S^2 \C^4$. 
Moreover notice that the two summands in (\ref{two summands}) above correspond to a $\PP^9$ and a $\PP^4$ disjoint 
in $\PP^{14}=\PP(\Lambda^2(\Lambda^2 \C^4))$. 

We have a Grassmannian $\G(2,5) \hookrightarrow \PP(\Lambda^2 (\Lambda_0^2 \C^4))=\PP(\Lambda^2\C^5)$. Moreover remark that 
the $\PP^4=\PP(\Lambda_0^2 \C^4)$ is entirely contained in $\G(2,6)$, simply because all elements of $\Lambda_0^2 \C^4$ will be rank-2 
tensors in $\PP(\Lambda^2 (\Lambda^2 \C^4))$.

This means that we can identify $\phi$ with the projection $\G(2,6) \dashrightarrow \PP^9=\PP( \Lambda_0^2 \C^4)$ from $\PP^4=\PP(\Lambda_0^2 \C^4)$. 
But since $\PP^4 \subset \G(2,6)$, the projection is induced by the projection $\PP^5 \dashrightarrow \PP^4$ and necessarily maps $\G(2,6)$ 
in the Grassmannian $\G(2,5)$, concluding our proof.

In \cite[Proposition 4.1]{note_vanni} it is shown that for $n=4$ the $\im \phi$ coincides with the set of Hamiltonian square roots of the identity times a scalar. 

As an addendum, we point out that the image of $\phi$ the union of the orbits $C$, $G$ and $H$ listed in Table I of \cite{ott_szu}.

\item[c)] Let $S \in \PP(S^2 \C^6)$ be a symmetric matrix. The determinant of $S-tJ$ has the form:
$$\det(S-tJ)=t^6+\gamma_4(S)t^4+\gamma_2(S)t^2+\det(S),$$

with only even terms. The image of $\phi$ is a hypersurface of degree 4 with equation $\gamma_4^2+4\gamma_2=0$. 

This can be proved either by direct computation, computing all possible normal form of the matrices in $\im \phi$, as is done in \cite[Proposition 4.2]{note_vanni}, or else checked with the computer algebra system Macaulay2 \cite{M2}. To understand where does the equation come from, notice that Lemma \ref{up to sp} together with the equivariancy of the map $\phi$ guarantees that, 
up to the symplectic action, we can take the matrix $B$ to be of the form $B=\left[ \begin{array}{cc}
0&\beta\\
-\tr{\beta}&0
\end{array}
\right]$. Let us suppose now that $A$ is also in the block form $A=\left[ \begin{array}{cc}
0&\alpha\\
-\alpha^T&0
\end{array}
\right]$. 

Then $AJB-BJA=\left[ \begin{array}{cc}
0&{[\beta,\alpha]}\\
{[\beta,\alpha]^T}&0
\end{array}
\right]$ 
and we can compute:
\begin{align}
\nonumber \det((AJB-BJA)-tJ)&=\det([\beta,\alpha]-tI)\det([\beta,\alpha]^T+tI)\\
\nonumber &=(-t^3+at+b)(t^3-at+b)\\
\nonumber &=-t^6+2at^4-a^2x^2+b^2,
\end{align}

for some coefficients $a$ and $b$. 

\item[d)] We perform a local computation and then induction on $p=n/2$. A different proof of the fact that $\phi$ is dominant can be 
found in \cite[Proposition 4.4]{note_vanni}, where the author proves that for $n \ge 8$ the set of symmetric matrices with all distinct eigenvalues is inside 
$\im \phi$. 

So let us start our induction. For the low values of $p$ that are the starting point of our induction, namely $p=4$ and $5$, the statement 
can be checked directly with the aid of the computer algebra system Macaulay2 \cite{M2}. 

Assume now $p \ge 6$. Let $M \in \sk$ and consider the submatrices $M^1$ and $M_2$ obtained by removing respectively the $p$-th and the $2p$-th rows and columns and 
the first and $(p+1)$-th rows and columns, as illustrated in the following Figure 4:

\begin{center}
\scalebox{0.65} % Change this value to rescale the drawing.
{
\begin{pspicture}(0,-3.546)(16.046,3.546)
\psframe[linewidth=0.05,dimen=outer](16.04,3.54)(8.96,-3.54)
\psline[linewidth=0.051999997cm](12.5,3.52)(12.5,-3.52)
\psline[linewidth=0.051999997cm](8.98,0.0)(16.02,0.0)
\psline[linewidth=0.0139999995cm](8.98,3.2)(16.02,3.2)
\psline[linewidth=0.0139999995cm](8.98,-0.32)(16.02,-0.32)
\psline[linewidth=0.0139999995cm](9.3,3.52)(9.3,-3.52)
\psline[linewidth=0.0139999995cm](12.82,3.52)(12.82,-3.52)
\psdots[dotsize=0.18,dotstyle=x](2.42,0.16)
\psdots[dotsize=0.18,dotstyle=x](2.42,-3.36)
\psdots[dotsize=0.18,dotstyle=x](2.74,-3.36)
\psdots[dotsize=0.18,dotstyle=x](3.06,-3.36)
\psdots[dotsize=0.18,dotstyle=x](3.38,-3.36)
\psdots[dotsize=0.18,dotstyle=x](3.7,-3.36)
\psdots[dotsize=0.18,dotstyle=x](4.02,-3.36)
\psdots[dotsize=0.18,dotstyle=x](4.34,-3.36)
\psdots[dotsize=0.18,dotstyle=x](4.66,-3.36)
\psdots[dotsize=0.18,dotstyle=x](4.98,-3.36)
\psdots[dotsize=0.18,dotstyle=x](5.3,-3.36)
\psdots[dotsize=0.18,dotstyle=x](5.62,-3.36)
\psdots[dotsize=0.18,dotstyle=x](5.94,-3.36)
\psdots[dotsize=0.18,dotstyle=x](5.94,0.16)
\psdots[dotsize=0.18,dotstyle=x](2.74,0.16)
\psdots[dotsize=0.18,dotstyle=x](3.06,0.16)
\psdots[dotsize=0.18,dotstyle=x](3.38,0.16)
\psdots[dotsize=0.18,dotstyle=x](3.7,0.16)
\psdots[dotsize=0.18,dotstyle=x](4.02,0.16)
\psdots[dotsize=0.18,dotstyle=x](4.34,0.16)
\psdots[dotsize=0.18,dotstyle=x](4.66,0.16)
\psdots[dotsize=0.18,dotstyle=x](4.98,0.16)
\psdots[dotsize=0.18,dotstyle=x](5.3,0.16)
\psdots[dotsize=0.18,dotstyle=x](5.62,0.16)
\psdots[dotsize=0.18,dotstyle=x](0.18,0.16)
\psdots[dotsize=0.18,dotstyle=x](0.5,0.16)
\psdots[dotsize=0.18,dotstyle=x](0.82,0.16)
\psdots[dotsize=0.18,dotstyle=x](1.14,0.16)
\psdots[dotsize=0.18,dotstyle=x](1.46,0.16)
\psdots[dotsize=0.18,dotstyle=x](1.78,0.16)
\psdots[dotsize=0.18,dotstyle=x](2.1,0.16)
\psdots[dotsize=0.18,dotstyle=x](0.18,-3.36)
\psdots[dotsize=0.18,dotstyle=x](0.5,-3.36)
\psdots[dotsize=0.18,dotstyle=x](0.82,-3.36)
\psdots[dotsize=0.18,dotstyle=x](1.14,-3.36)
\psdots[dotsize=0.18,dotstyle=x](1.46,-3.36)
\psdots[dotsize=0.18,dotstyle=x](1.78,-3.36)
\psdots[dotsize=0.18,dotstyle=x](2.1,-3.36)
\psdots[dotsize=0.18,dotstyle=x](9.14,3.36)
\psdots[dotsize=0.18,dotstyle=x](9.46,3.36)
\psdots[dotsize=0.18,dotstyle=x](9.78,3.36)
\psdots[dotsize=0.18,dotstyle=x](10.1,3.36)
\psdots[dotsize=0.18,dotstyle=x](10.42,3.36)
\psdots[dotsize=0.18,dotstyle=x](10.74,3.36)
\psdots[dotsize=0.18,dotstyle=x](11.06,3.36)
\psdots[dotsize=0.18,dotstyle=x](11.38,3.36)
\psdots[dotsize=0.18,dotstyle=x](11.7,3.36)
\psdots[dotsize=0.18,dotstyle=x](12.02,3.36)
\psdots[dotsize=0.18,dotstyle=x](12.34,3.36)
\psdots[dotsize=0.18,dotstyle=x](12.66,3.36)
\psdots[dotsize=0.18,dotstyle=x](12.98,3.36)
\psdots[dotsize=0.18,dotstyle=x](13.3,3.36)
\psdots[dotsize=0.18,dotstyle=x](13.62,3.36)
\psdots[dotsize=0.18,dotstyle=x](13.94,3.36)
\psdots[dotsize=0.18,dotstyle=x](14.26,3.36)
\psdots[dotsize=0.18,dotstyle=x](14.58,3.36)
\psdots[dotsize=0.18,dotstyle=x](14.9,3.36)
\psdots[dotsize=0.18,dotstyle=x](15.22,3.36)
\psdots[dotsize=0.18,dotstyle=x](15.54,3.36)
\psdots[dotsize=0.18,dotstyle=x](15.86,3.36)
\psdots[dotsize=0.18,dotstyle=x](9.14,3.04)
\psdots[dotsize=0.18,dotstyle=x](9.14,2.72)
\psdots[dotsize=0.18,dotstyle=x](9.14,2.4)
\psdots[dotsize=0.18,dotstyle=x](9.14,2.08)
\psdots[dotsize=0.18,dotstyle=x](9.14,1.76)
\psdots[dotsize=0.18,dotstyle=x](9.14,1.44)
\psdots[dotsize=0.18,dotstyle=x](9.14,1.12)
\psdots[dotsize=0.18,dotstyle=x](9.14,0.8)
\psdots[dotsize=0.18,dotstyle=x](9.14,0.48)
\psdots[dotsize=0.18,dotstyle=x](9.14,0.16)
\psdots[dotsize=0.18,dotstyle=x](12.66,3.04)
\psdots[dotsize=0.18,dotstyle=x](12.66,2.72)
\psdots[dotsize=0.18,dotstyle=x](12.66,2.4)
\psdots[dotsize=0.18,dotstyle=x](12.66,2.08)
\psdots[dotsize=0.18,dotstyle=x](12.66,1.76)
\psdots[dotsize=0.18,dotstyle=x](12.66,1.44)
\psdots[dotsize=0.18,dotstyle=x](12.66,1.12)
\psdots[dotsize=0.18,dotstyle=x](12.66,0.8)
\psdots[dotsize=0.18,dotstyle=x](12.66,0.48)
\psdots[dotsize=0.18,dotstyle=x](12.66,0.16)
\psdots[dotsize=0.18,dotstyle=x](12.66,-0.16)
\psdots[dotsize=0.18,dotstyle=x](12.98,-0.16)
\psdots[dotsize=0.18,dotstyle=x](13.3,-0.16)
\psdots[dotsize=0.18,dotstyle=x](13.62,-0.16)
\psdots[dotsize=0.18,dotstyle=x](9.14,-0.16)
\psdots[dotsize=0.18,dotstyle=x](9.46,-0.16)
\psdots[dotsize=0.18,dotstyle=x](9.78,-0.16)
\psdots[dotsize=0.18,dotstyle=x](10.1,-0.16)
\psdots[dotsize=0.18,dotstyle=x](10.42,-0.16)
\psdots[dotsize=0.18,dotstyle=x](9.14,-0.48)
\psdots[dotsize=0.18,dotstyle=x](12.66,-3.36)
\psdots[dotsize=0.18,dotstyle=x](12.66,-3.04)
\psdots[dotsize=0.18,dotstyle=x](12.66,-2.72)
\psdots[dotsize=0.18,dotstyle=x](12.66,-2.4)
\psdots[dotsize=0.18,dotstyle=x](12.66,-2.08)
\psdots[dotsize=0.18,dotstyle=x](12.66,-1.76)
\psdots[dotsize=0.18,dotstyle=x](12.66,-1.44)
\psdots[dotsize=0.18,dotstyle=x](12.66,-1.12)
\psdots[dotsize=0.18,dotstyle=x](12.66,-0.8)
\psdots[dotsize=0.18,dotstyle=x](12.66,-0.48)
\psdots[dotsize=0.18,dotstyle=x](13.94,-0.16)
\psdots[dotsize=0.18,dotstyle=x](14.26,-0.16)
\psdots[dotsize=0.18,dotstyle=x](14.58,-0.16)
\psdots[dotsize=0.18,dotstyle=x](14.9,-0.16)
\psdots[dotsize=0.18,dotstyle=x](15.22,-0.16)
\psdots[dotsize=0.18,dotstyle=x](15.54,-0.16)
\psdots[dotsize=0.18,dotstyle=x](15.86,-0.16)
\psdots[dotsize=0.18,dotstyle=x](12.34,-0.16)
\psdots[dotsize=0.18,dotstyle=x](10.74,-0.16)
\psdots[dotsize=0.18,dotstyle=x](11.06,-0.16)
\psdots[dotsize=0.18,dotstyle=x](11.38,-0.16)
\psdots[dotsize=0.18,dotstyle=x](11.7,-0.16)
\psdots[dotsize=0.18,dotstyle=x](12.02,-0.16)
\psdots[dotsize=0.18,dotstyle=x](9.14,-0.8)
\psdots[dotsize=0.18,dotstyle=x](9.14,-1.12)
\psdots[dotsize=0.18,dotstyle=x](9.14,-1.44)
\psdots[dotsize=0.18,dotstyle=x](9.14,-1.76)
\psdots[dotsize=0.18,dotstyle=x](9.14,-2.08)
\psdots[dotsize=0.18,dotstyle=x](9.14,-2.4)
\psdots[dotsize=0.18,dotstyle=x](9.14,-2.72)
\psdots[dotsize=0.18,dotstyle=x](9.14,-3.04)
\psdots[dotsize=0.18,dotstyle=x](9.14,-3.36)
\psframe[linewidth=0.05,dimen=outer](7.08,3.54)(0.0,-3.54)
\psline[linewidth=0.05cm](3.54,3.52)(3.54,-3.52)
\psline[linewidth=0.05cm](0.02,0.0)(7.06,0.0)
\psline[linewidth=0.0139999995cm](0.02,0.32)(7.06,0.32)
\psline[linewidth=0.0139999995cm](0.02,-3.2)(7.06,-3.2)
\psline[linewidth=0.0139999995cm](3.22,3.52)(3.22,-3.52)
\psline[linewidth=0.0139999995cm](6.74,3.52)(6.74,-3.52)
\psdots[dotsize=0.18,dotstyle=x](3.38,2.4)
\psdots[dotsize=0.18,dotstyle=x](3.38,2.72)
\psdots[dotsize=0.18,dotstyle=x](3.38,3.04)
\psdots[dotsize=0.18,dotstyle=x](3.38,3.36)
\psdots[dotsize=0.18,dotstyle=x](3.38,2.08)
\psdots[dotsize=0.18,dotstyle=x](3.38,1.76)
\psdots[dotsize=0.18,dotstyle=x](3.38,1.44)
\psdots[dotsize=0.18,dotstyle=x](3.38,1.12)
\psdots[dotsize=0.18,dotstyle=x](3.38,0.8)
\psdots[dotsize=0.18,dotstyle=x](3.38,0.48)
\psdots[dotsize=0.18,dotstyle=x](3.38,0.16)
\psdots[dotsize=0.18,dotstyle=x](3.38,-0.16)
\psdots[dotsize=0.18,dotstyle=x](3.38,-0.48)
\psdots[dotsize=0.18,dotstyle=x](3.38,-0.8)
\psdots[dotsize=0.18,dotstyle=x](3.38,-1.12)
\psdots[dotsize=0.18,dotstyle=x](3.38,-1.44)
\psdots[dotsize=0.18,dotstyle=x](3.38,-1.76)
\psdots[dotsize=0.18,dotstyle=x](3.38,-2.08)
\psdots[dotsize=0.18,dotstyle=x](3.38,-2.4)
\psdots[dotsize=0.18,dotstyle=x](3.38,-2.72)
\psdots[dotsize=0.18,dotstyle=x](3.38,-3.04)
\psdots[dotsize=0.18,dotstyle=x](3.38,-3.36)
\psdots[dotsize=0.18,dotstyle=x](3.7,-3.36)
\psdots[dotsize=0.18,dotstyle=x](4.02,-3.36)
\psdots[dotsize=0.18,dotstyle=x](4.34,-3.36)
\psdots[dotsize=0.18,dotstyle=x](4.66,-3.36)
\psdots[dotsize=0.18,dotstyle=x](4.98,-3.36)
\psdots[dotsize=0.18,dotstyle=x](5.3,-3.36)
\psdots[dotsize=0.18,dotstyle=x](5.62,-3.36)
\psdots[dotsize=0.18,dotstyle=x](5.94,-3.36)
\psdots[dotsize=0.18,dotstyle=x](6.26,-3.36)
\psdots[dotsize=0.18,dotstyle=x](6.58,-3.36)
\psdots[dotsize=0.18,dotstyle=x](6.9,-3.36)
\psdots[dotsize=0.18,dotstyle=x](6.9,-3.04)
\psdots[dotsize=0.18,dotstyle=x](6.9,-2.72)
\psdots[dotsize=0.18,dotstyle=x](6.9,-2.4)
\psdots[dotsize=0.18,dotstyle=x](6.9,-2.08)
\psdots[dotsize=0.18,dotstyle=x](6.9,-1.76)
\psdots[dotsize=0.18,dotstyle=x](6.9,-1.44)
\psdots[dotsize=0.18,dotstyle=x](6.9,-1.12)
\psdots[dotsize=0.18,dotstyle=x](6.9,-0.8)
\psdots[dotsize=0.18,dotstyle=x](6.9,-0.48)
\psdots[dotsize=0.18,dotstyle=x](6.9,-0.16)
\psdots[dotsize=0.18,dotstyle=x](6.9,0.16)
\psdots[dotsize=0.18,dotstyle=x](6.9,0.48)
\psdots[dotsize=0.18,dotstyle=x](6.9,0.8)
\psdots[dotsize=0.18,dotstyle=x](6.9,1.12)
\psdots[dotsize=0.18,dotstyle=x](6.9,1.44)
\psdots[dotsize=0.18,dotstyle=x](6.9,1.76)
\psdots[dotsize=0.18,dotstyle=x](6.9,2.08)
\psdots[dotsize=0.18,dotstyle=x](6.9,2.4)
\psdots[dotsize=0.18,dotstyle=x](6.9,2.72)
\psdots[dotsize=0.18,dotstyle=x](6.9,3.04)
\psdots[dotsize=0.18,dotstyle=x](6.9,3.36)
\psdots[dotsize=0.18,dotstyle=x](3.7,0.16)
\psdots[dotsize=0.18,dotstyle=x](4.02,0.16)
\psdots[dotsize=0.18,dotstyle=x](4.34,0.16)
\psdots[dotsize=0.18,dotstyle=x](4.66,0.16)
\psdots[dotsize=0.18,dotstyle=x](4.98,0.16)
\psdots[dotsize=0.18,dotstyle=x](5.3,0.16)
\psdots[dotsize=0.18,dotstyle=x](5.62,0.16)
\psdots[dotsize=0.18,dotstyle=x](5.94,0.16)
\psdots[dotsize=0.18,dotstyle=x](6.26,0.16)
\psdots[dotsize=0.18,dotstyle=x](6.58,0.16)
\psdots[dotsize=0.18,dotstyle=x](0.18,0.16)
\psdots[dotsize=0.18,dotstyle=x](0.5,0.16)
\psdots[dotsize=0.18,dotstyle=x](0.82,0.16)
\psdots[dotsize=0.18,dotstyle=x](1.14,0.16)
\psdots[dotsize=0.18,dotstyle=x](1.46,0.16)
\psdots[dotsize=0.18,dotstyle=x](1.78,0.16)
\psdots[dotsize=0.18,dotstyle=x](2.1,0.16)
\psdots[dotsize=0.18,dotstyle=x](2.42,0.16)
\psdots[dotsize=0.18,dotstyle=x](2.74,0.16)
\psdots[dotsize=0.18,dotstyle=x](3.06,0.16)
\psdots[dotsize=0.18,dotstyle=x](0.18,-3.36)
\psdots[dotsize=0.18,dotstyle=x](0.5,-3.36)
\psdots[dotsize=0.18,dotstyle=x](0.82,-3.36)
\psdots[dotsize=0.18,dotstyle=x](1.14,-3.36)
\psdots[dotsize=0.18,dotstyle=x](1.46,-3.36)
\psdots[dotsize=0.18,dotstyle=x](1.78,-3.36)
\psdots[dotsize=0.18,dotstyle=x](2.1,-3.36)
\psdots[dotsize=0.18,dotstyle=x](2.42,-3.36)
\psdots[dotsize=0.18,dotstyle=x](2.74,-3.36)
\psdots[dotsize=0.18,dotstyle=x](3.06,-3.36)
% \usefont{T1}{ptm}{m}{n}
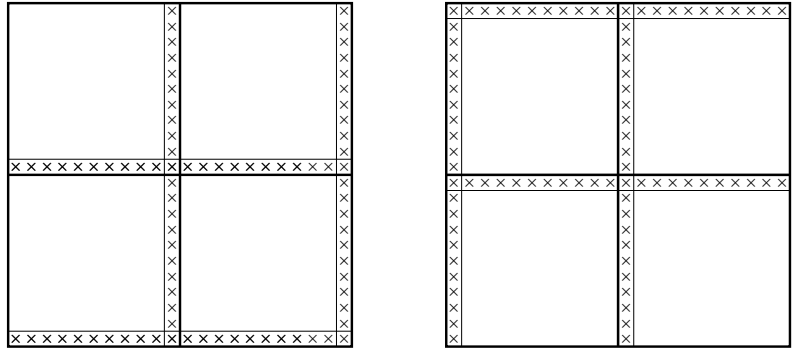
\rput(8.5,-4.5){\Large{Figure 4. Shape of the matrices $M^1$ (left) and $M_2$ (right)}}
\end{pspicture} 
}
\end{center}

\vskip.3in
              
Notice that $M^1$ and $M_2$ are $(n-1)-$dimensional skew-symmetric matrices and that the symplectic form $J$ is ``preserved'' under this cropping.
Moreover the submatrix obtained by removing from $M$ all eight rows and columns combined is an $(n-2)-$dimensional skew-symmetric matrix as well and the symplectic 
form $J$ is ``preserved'' in this case, too.
Call this submatrix $M_2^1$. It has the form illustrated in the following Figure 5:

\begin{center}
\scalebox{0.8} 
{
\begin{pspicture}(0,-3.545)(7.085,3.545)
\definecolor{color12}{rgb}{1.0,0.2,0.0}
\psdots[dotsize=0.18,dotstyle=x](2.42,0.16)
\psdots[dotsize=0.18,dotstyle=x](2.42,-3.36)
\psdots[dotsize=0.18,dotstyle=x](2.74,-3.36)
\psdots[dotsize=0.18,dotstyle=x](3.06,-3.36)
\psdots[dotsize=0.18,dotstyle=x](3.38,-3.36)
\psdots[dotsize=0.18,dotstyle=otimes](3.7,-3.36)
\psdots[dotsize=0.18,dotstyle=x](4.02,-3.36)
\psdots[dotsize=0.18,dotstyle=x](4.34,-3.36)
\psdots[dotsize=0.18,dotstyle=x](4.66,-3.36)
\psdots[dotsize=0.18,dotstyle=x](4.98,-3.36)
\psdots[dotsize=0.18,dotstyle=x](5.3,-3.36)
\psdots[dotsize=0.18,dotstyle=x](5.62,-3.36)
\psdots[dotsize=0.18,dotstyle=x](5.94,-3.36)
\psdots[dotsize=0.18,dotstyle=x](5.94,0.16)
\psdots[dotsize=0.18,dotstyle=x](2.74,0.16)
\psdots[dotsize=0.18,dotstyle=x](3.06,0.16)
\psdots[dotsize=0.18,dotstyle=x](3.38,0.16)
\psdots[dotsize=0.18,dotstyle=otimes](3.7,0.16)
\psdots[dotsize=0.18,dotstyle=x](4.02,0.16)
\psdots[dotsize=0.18,dotstyle=x](4.34,0.16)
\psdots[dotsize=0.18,dotstyle=x](4.66,0.16)
\psdots[dotsize=0.18,dotstyle=x](4.98,0.16)
\psdots[dotsize=0.18,dotstyle=x](5.3,0.16)
\psdots[dotsize=0.18,dotstyle=x](5.62,0.16)
\psdots[dotsize=0.18,dotstyle=otimes](0.18,0.16)
\psdots[dotsize=0.18,dotstyle=x](0.5,0.16)
\psdots[dotsize=0.18,dotstyle=x](0.82,0.16)
\psdots[dotsize=0.18,dotstyle=x](1.14,0.16)
\psdots[dotsize=0.18,dotstyle=x](1.46,0.16)
\psdots[dotsize=0.18,dotstyle=x](1.78,0.16)
\psdots[dotsize=0.18,dotstyle=x](2.1,0.16)
\psdots[dotsize=0.18,dotstyle=otimes](0.18,-3.36)
\psdots[dotsize=0.18,dotstyle=x](0.5,-3.36)
\psdots[dotsize=0.18,dotstyle=x](0.82,-3.36)
\psdots[dotsize=0.18,dotstyle=x](1.14,-3.36)
\psdots[dotsize=0.18,dotstyle=x](1.46,-3.36)
\psdots[dotsize=0.18,dotstyle=x](1.78,-3.36)
\psdots[dotsize=0.18,dotstyle=x](2.1,-3.36)
\psframe[linewidth=0.05,dimen=outer](7.08,3.54)(0.0,-3.54)
\psline[linewidth=0.05cm](3.54,3.52)(3.54,-3.52)
\psline[linewidth=0.05cm](0.02,0.0)(7.06,0.0)
\psline[linewidth=0.0139999995cm](0.02,0.32)(7.06,0.32)
\psline[linewidth=0.0139999995cm](0.02,-3.2)(7.06,-3.2)
\psline[linewidth=0.0139999995cm](3.22,3.52)(3.22,-3.52)
\psline[linewidth=0.0139999995cm](6.74,3.52)(6.74,-3.52)
\psdots[dotsize=0.18,dotstyle=x](3.38,2.4)
\psdots[dotsize=0.18,dotstyle=x](3.38,2.72)
\psdots[dotsize=0.18,dotstyle=x](3.38,3.04)
\psdots[dotsize=0.19,linecolor=color12,dotstyle=otimes](3.38,3.36)
\psdots[dotsize=0.18,dotstyle=x](3.38,2.08)
\psdots[dotsize=0.18,dotstyle=x](3.38,1.76)
\psdots[dotsize=0.18,dotstyle=x](3.38,1.44)
\psdots[dotsize=0.18,dotstyle=x](3.38,1.12)
\psdots[dotsize=0.18,dotstyle=x](3.38,0.8)
\psdots[dotsize=0.18,dotstyle=x](3.38,0.48)
\psdots[dotsize=0.18,dotstyle=x](3.38,0.16)
\psdots[dotsize=0.19,linecolor=color12,dotstyle=otimes](3.38,-0.16)
\psdots[dotsize=0.18,dotstyle=x](3.38,-0.48)
\psdots[dotsize=0.18,dotstyle=x](3.38,-0.8)
\psdots[dotsize=0.18,dotstyle=x](3.38,-1.12)
\psdots[dotsize=0.18,dotstyle=x](3.38,-1.44)
\psdots[dotsize=0.18,dotstyle=x](3.38,-1.76)
\psdots[dotsize=0.18,dotstyle=x](3.38,-2.08)
\psdots[dotsize=0.18,dotstyle=x](3.38,-2.4)
\psdots[dotsize=0.18,dotstyle=x](3.38,-2.72)
\psdots[dotsize=0.18,dotstyle=x](3.38,-3.04)
\psdots[dotsize=0.18,dotstyle=x](3.38,-3.36)
\psdots[dotsize=0.19,linecolor=color12,dotstyle=otimes](3.7,-3.36)
\psdots[dotsize=0.18,dotstyle=x](4.02,-3.36)
\psdots[dotsize=0.18,dotstyle=x](4.34,-3.36)
\psdots[dotsize=0.18,dotstyle=x](4.66,-3.36)
\psdots[dotsize=0.18,dotstyle=x](4.98,-3.36)
\psdots[dotsize=0.18,dotstyle=x](5.3,-3.36)
\psdots[dotsize=0.18,dotstyle=x](5.62,-3.36)
\psdots[dotsize=0.18,dotstyle=x](5.94,-3.36)
\psdots[dotsize=0.18,dotstyle=x](6.26,-3.36)
\psdots[dotsize=0.18,dotstyle=x](6.58,-3.36)
\psdots[dotsize=0.18,dotstyle=x](6.9,-3.36)
\psdots[dotsize=0.18,dotstyle=x](6.9,-3.04)
\psdots[dotsize=0.18,dotstyle=x](6.9,-2.72)
\psdots[dotsize=0.18,dotstyle=x](6.9,-2.4)
\psdots[dotsize=0.18,dotstyle=x](6.9,-2.08)
\psdots[dotsize=0.18,dotstyle=x](6.9,-1.76)
\psdots[dotsize=0.18,dotstyle=x](6.9,-1.44)
\psdots[dotsize=0.18,dotstyle=x](6.9,-1.12)
\psdots[dotsize=0.18,dotstyle=x](6.9,-0.8)
\psdots[dotsize=0.18,dotstyle=x](6.9,-0.48)
\psdots[dotsize=0.19,linecolor=color12,dotstyle=otimes](6.9,-0.16)
\psdots[dotsize=0.18,dotstyle=x](6.9,0.16)
\psdots[dotsize=0.18,dotstyle=x](6.9,0.48)
\psdots[dotsize=0.18,dotstyle=x](6.9,0.8)
\psdots[dotsize=0.18,dotstyle=x](6.9,1.12)
\psdots[dotsize=0.18,dotstyle=x](6.9,1.44)
\psdots[dotsize=0.18,dotstyle=x](6.9,1.76)
\psdots[dotsize=0.18,dotstyle=x](6.9,2.08)
\psdots[dotsize=0.18,dotstyle=x](6.9,2.4)
\psdots[dotsize=0.18,dotstyle=x](6.9,2.72)
\psdots[dotsize=0.18,dotstyle=x](6.9,3.04)
\psdots[dotsize=0.19,linecolor=color12,dotstyle=otimes](6.9,3.36)
\psdots[dotsize=0.19,linecolor=color12,dotstyle=otimes](3.7,0.16)
\psdots[dotsize=0.18,dotstyle=x](4.02,0.16)
\psdots[dotsize=0.18,dotstyle=x](4.34,0.16)
\psdots[dotsize=0.18,dotstyle=x](4.66,0.16)
\psdots[dotsize=0.18,dotstyle=x](4.98,0.16)
\psdots[dotsize=0.18,dotstyle=x](5.3,0.16)
\psdots[dotsize=0.18,dotstyle=x](5.62,0.16)
\psdots[dotsize=0.18,dotstyle=x](5.94,0.16)
\psdots[dotsize=0.18,dotstyle=x](6.26,0.16)
\psdots[dotsize=0.18,dotstyle=x](6.58,0.16)
\psdots[dotsize=0.19,linecolor=color12,dotstyle=otimes](0.18,0.16)
\psdots[dotsize=0.18,dotstyle=x](0.5,0.16)
\psdots[dotsize=0.18,dotstyle=x](0.82,0.16)
\psdots[dotsize=0.18,dotstyle=x](1.14,0.16)
\psdots[dotsize=0.18,dotstyle=x](1.46,0.16)
\psdots[dotsize=0.18,dotstyle=x](1.78,0.16)
\psdots[dotsize=0.18,dotstyle=x](2.1,0.16)
\psdots[dotsize=0.18,dotstyle=x](2.42,0.16)
\psdots[dotsize=0.18,dotstyle=x](2.74,0.16)
\psdots[dotsize=0.18,dotstyle=x](3.06,0.16)
\psdots[dotsize=0.19,linecolor=color12,dotstyle=otimes](0.18,-3.36)
\psdots[dotsize=0.18,dotstyle=x](0.5,-3.36)
\psdots[dotsize=0.18,dotstyle=x](0.82,-3.36)
\psdots[dotsize=0.18,dotstyle=x](1.14,-3.36)
\psdots[dotsize=0.18,dotstyle=x](1.46,-3.36)
\psdots[dotsize=0.18,dotstyle=x](1.78,-3.36)
\psdots[dotsize=0.18,dotstyle=x](2.1,-3.36)
\psdots[dotsize=0.18,dotstyle=x](2.42,-3.36)
\psdots[dotsize=0.18,dotstyle=x](2.74,-3.36)
\psdots[dotsize=0.18,dotstyle=x](3.06,-3.36)
\psline[linewidth=0.0139999995cm](0.02,3.2)(7.06,3.2)
\psline[linewidth=0.0139999995cm](0.02,-0.32)(7.06,-0.32)
\psline[linewidth=0.0139999995cm](3.86,3.52)(3.86,-3.52)
\psline[linewidth=0.0139999995cm](0.34,3.52)(0.34,-3.52)
\psdots[dotsize=0.18,dotstyle=x](3.7,1.12)
\psdots[dotsize=0.18,dotstyle=x](3.7,3.36)
\psdots[dotsize=0.18,dotstyle=x](4.02,3.36)
\psdots[dotsize=0.18,dotstyle=x](4.34,3.36)
\psdots[dotsize=0.18,dotstyle=x](4.66,3.36)
\psdots[dotsize=0.18,dotstyle=x](4.98,3.36)
\psdots[dotsize=0.18,dotstyle=x](5.3,3.36)
\psdots[dotsize=0.18,dotstyle=x](5.62,3.36)
\psdots[dotsize=0.18,dotstyle=x](5.94,3.36)
\psdots[dotsize=0.18,dotstyle=x](6.26,3.36)
\psdots[dotsize=0.18,dotstyle=x](6.58,3.36)
\psdots[dotsize=0.18,dotstyle=x](0.18,3.36)
\psdots[dotsize=0.18,dotstyle=x](0.5,3.36)
\psdots[dotsize=0.18,dotstyle=x](0.82,3.36)
\psdots[dotsize=0.18,dotstyle=x](1.14,3.36)
\psdots[dotsize=0.18,dotstyle=x](1.46,3.36)
\psdots[dotsize=0.18,dotstyle=x](1.78,3.36)
\psdots[dotsize=0.18,dotstyle=x](2.1,3.36)
\psdots[dotsize=0.18,dotstyle=x](2.42,3.36)
\psdots[dotsize=0.18,dotstyle=x](2.74,3.36)
\psdots[dotsize=0.18,dotstyle=x](3.06,3.36)
\psdots[dotsize=0.18,dotstyle=x](3.7,3.04)
\psdots[dotsize=0.18,dotstyle=x](3.7,2.72)
\psdots[dotsize=0.18,dotstyle=x](3.7,2.4)
\psdots[dotsize=0.18,dotstyle=x](3.7,2.08)
\psdots[dotsize=0.18,dotstyle=x](3.7,1.76)
\psdots[dotsize=0.18,dotstyle=x](3.7,1.44)
\psdots[dotsize=0.18,dotstyle=x](3.7,0.8)
\psdots[dotsize=0.18,dotstyle=x](3.7,0.48)
\psdots[dotsize=0.18,dotstyle=x](3.7,-0.16)
\psdots[dotsize=0.18,dotstyle=x](3.7,-0.48)
\psdots[dotsize=0.18,dotstyle=x](3.7,-0.8)
\psdots[dotsize=0.18,dotstyle=x](3.7,-1.12)
\psdots[dotsize=0.18,dotstyle=x](4.02,-0.16)
\psdots[dotsize=0.18,dotstyle=x](4.34,-0.16)
\psdots[dotsize=0.18,dotstyle=x](4.66,-0.16)
\psdots[dotsize=0.18,dotstyle=x](4.98,-0.16)
\psdots[dotsize=0.18,dotstyle=x](5.3,-0.16)
\psdots[dotsize=0.18,dotstyle=x](5.62,-0.16)
\psdots[dotsize=0.18,dotstyle=x](5.94,-0.16)
\psdots[dotsize=0.18,dotstyle=x](6.26,-0.16)
\psdots[dotsize=0.18,dotstyle=x](6.58,-0.16)
\psdots[dotsize=0.18,dotstyle=x](3.7,-1.44)
\psdots[dotsize=0.18,dotstyle=x](3.7,-1.76)
\psdots[dotsize=0.18,dotstyle=x](3.7,-2.08)
\psdots[dotsize=0.18,dotstyle=x](3.7,-2.4)
\psdots[dotsize=0.18,dotstyle=x](3.7,-2.72)
\psdots[dotsize=0.18,dotstyle=x](3.7,-3.04)
\psdots[dotsize=0.18,dotstyle=x](0.18,-3.04)
\psdots[dotsize=0.18,dotstyle=x](0.18,-2.72)
\psdots[dotsize=0.18,dotstyle=x](0.18,-2.4)
\psdots[dotsize=0.18,dotstyle=x](0.18,-0.16)
\psdots[dotsize=0.18,dotstyle=x](0.5,-0.16)
\psdots[dotsize=0.18,dotstyle=x](0.82,-0.16)
\psdots[dotsize=0.18,dotstyle=x](1.14,-0.16)
\psdots[dotsize=0.18,dotstyle=x](1.46,-0.16)
\psdots[dotsize=0.18,dotstyle=x](1.78,-0.16)
\psdots[dotsize=0.18,dotstyle=x](2.1,-0.16)
\psdots[dotsize=0.18,dotstyle=x](2.42,-0.16)
\psdots[dotsize=0.18,dotstyle=x](2.74,-0.16)
\psdots[dotsize=0.18,dotstyle=x](3.06,-0.16)
\psdots[dotsize=0.18,dotstyle=x](0.18,3.04)
\psdots[dotsize=0.18,dotstyle=x](0.18,2.72)
\psdots[dotsize=0.18,dotstyle=x](0.18,2.4)
\psdots[dotsize=0.18,dotstyle=x](0.18,2.08)
\psdots[dotsize=0.18,dotstyle=x](0.18,1.76)
\psdots[dotsize=0.18,dotstyle=x](0.18,1.44)
\psdots[dotsize=0.18,dotstyle=x](0.18,1.12)
\psdots[dotsize=0.18,dotstyle=x](0.18,0.8)
\psdots[dotsize=0.18,dotstyle=x](0.18,0.48)
\psdots[dotsize=0.18,dotstyle=x](0.18,-0.48)
\psdots[dotsize=0.18,dotstyle=x](0.18,-0.8)
\psdots[dotsize=0.18,dotstyle=x](0.18,-1.12)
\psdots[dotsize=0.18,dotstyle=x](0.18,-1.44)
\psdots[dotsize=0.18,dotstyle=x](0.18,-1.76)
\psdots[dotsize=0.18,dotstyle=x](0.18,-2.08)
\rput(3.6,-4.1){Figure 5. Shape of the matrices $M^1_2$}
\end{pspicture} 
}
\end{center}

\vskip.3in

If we consider the sum $M^1$+$M_2$, it gives back all the original matrix $M$ with the exception of four elements, namely $m_{1,p}$, $m_{1,n}$, $m_{p,p+1}$ and $m_{p+1,n}$ 
(and their four symmetric). We call the submatrix consisting of these four elements $M'$. 
The four elements of $M'$, together with their four symmetric ones, are identified with a different symbol and color in Figure 5.

Let us now consider the derivative of the map $\phi$, which is a linear map. Take $A$ and $B$ two general elements in $\sk \times \sk$. 
The Jacobian maps from $T_{A,B} \rightarrow \sy$ and is a 
$2\cdot {n \choose 2} \times {n+1 \choose 2}$ matrix. One has that: 

\begin{equation}\label{eq 1}
    \im T_{A,B}= \im T_{A^1,B^1} + \im T_{A_2,B_2}+ \im T_{A',B'}.
\end{equation}

On the other hand:
\begin{align}
\nonumber \dim(\im T_{A^1,B^1} + \im &T_{A_2,B_2})=\dim (\im T_{A^1,B^1}) + \\
\nonumber &+\dim(\im T_{A_2,B_2}) - \dim (\im T_{A^1_2,B^1_2}).
\end{align}

Now the fact that our cropping preserves the form $J$ implies that the map $\phi$ doesn't change the form of the matrices, meaning that for example $A^1JB^1-B^1JA^1$ 
is a symmetric matrix of the same form of $A^1$ and $B^1$ and the same holds for the other two types. Hence applying our inductive hypothesis:
$$\mbox{$\dim(\im T_{A^1,B^1} + \im T_{A_2,B_2})={n-1 \choose 2}+{n-1 \choose 2} - {n-3 \choose 2},$}$$
which once simplified becomes:

\begin{equation}\label{eq 2}
    \mbox{$\dim(\im T_{A^1,B^1} + \im T_{A_2,B_2})={n+1 \choose 2}-4.$}
\end{equation}

Combining (\ref{eq 1}) and (\ref{eq 2}) if we can prove that the four rows of the jacobian matrix of $\phi$ corresponding to the 
four elements of $A'$ and $B'$ are independent from the rest then the result will follow. 
But this is again a consequence of the fact that our cropping preserves the form $J$. 
Once cleverly ordered the rows of the jacobian matrix corresponding to $T_{A^1,B^1}$ and $T_{A_2,B_2}$ will have all zeros in the entries of the columns 
corresponding to $T_{A^1_2,B^1_2}$.
\qedhere
\end{itemize}
\end{proof}

% %%%%%%%%%%%%%%%%%%%%%%%%%%%%%%%%%%%%%%%%%%%%%%%%%%%%%%%%%%%%%%%%%%%%%%%%%%%%%%%%%
% \subsection{Counterexamples and open questions}
% %%%%%%%%%%%%%%%%%%%%%%%%%%%%%%%%%%%%%%%%%%%%%%%%%%%%%%%%%%%%%%%%%%%%%%%%%%%%%%%%%

\begin{rem} In light of the results of Proposition \ref{phi dominante}, it is quite natural to ask whether for big values of $n$ the map $\phi$ is in fact 
surjective and not only dominant. Even more interesting would be finding out whether the map $\phi$ composed with the projection $\PP(\sy) \twoheadrightarrow \PP(\syrc)$ 
is surjective for some value of the rank $r$. In other words, is there an $r$ for which all symmetric matrices of rank $r$ are of the form $AJB-BJA$ for a 
pair of skew-symmetric matrices $(A,B)$? Noferini \cite{note_vanni} has shown that the statement is false for $r=2$, any $n$. The question remains open for 
higher values of $r$. 

It is worth remarking that for any even rank $4 \le r \le n$ it is possible to exhibit regular 
skew-Hamiltonian matrices $JA$ and $JB$ whose commutator has rank $r$.
\end{rem}

%%%%%%%%%%%%%%%%%%%%%%%%%%%%%%%%%%%%%%%%%%%%%%%%%%%%%%%%%%%%%%%%%%%%%%%%%%%%%%%%%
%%%%%%%%%%%%%%%%%%%%%%%%%%%%%%%%%%%%%%%%%%%%%%%%%%%%%%%%%%%%%%%%%%%%%%%%%%%%%%%%%
\section{Concluding remarks}
%%%%%%%%%%%%%%%%%%%%%%%%%%%%%%%%%%%%%%%%%%%%%%%%%%%%%%%%%%%%%%%%%%%%%%%%%%%%%%%%%
%%%%%%%%%%%%%%%%%%%%%%%%%%%%%%%%%%%%%%%%%%%%%%%%%%%%%%%%%%%%%%%%%%%%%%%%%%%%%%%%%

%%%%%%%%%%%%%%%%%%%%%%%%%%%%%%%%%%%%%%%%%%%%%%%%%%%%%%%%%%%%%%%%%%%%%%%%%%%%%%%%%
\subsection{Linear monads}
%%%%%%%%%%%%%%%%%%%%%%%%%%%%%%%%%%%%%%%%%%%%%%%%%%%%%%%%%%%%%%%%%%%%%%%%%%%%%%%%%

Any element $E$ of $M(r,n)$ can be expressed as the cohomology bundle of a \emph{linear monad}, of type:
\begin{equation}\label{monade lineare}
V^* \otimes \OP(-1) \xrightarrow{\alpha} K \otimes \OP \xrightarrow{\beta} V \otimes \OP(1),
\end{equation}
where $V=\HH^1(E(-1))$ as usual and $K:=\HH^1(E \otimes \Omega_{\PP^2}^1)$ is a vector space of dimension $2n+r$, as proved in 
\cite[Lemma 1.4.2]{Hul}. 
To see this one only needs to repeat the proof of Lemma \ref{teorema 7.2} applying the ``dual Beilinson'' theorem, that is 
decomposing the bundle in $D^b(\PP^2)$ with respect to the exceptional collection $\langle \OP(-2), \OP(-1), \OP \rangle$ instead 
of the collection $\langle \OP(-1), \Omega^1_{\PP^2}(1), \OP \rangle$ . 

The unstructured case is treated by Hulek using monads of type \eqref{monade lineare}. 

This also means that, once we fix a framing, all elements of $M^0(r,n)$ are generalized instanton bundles in the sense of 
\cite{henni-jardim-martins}. 

Remark that by combining \eqref{monade lineare} and ($n$ copies of) the Euler sequence in the following diagram:

\begin{equation}\label{diagramma2monadi}
\xymatrix{&&0 \ar[d]&0\ar[d]&\\
&V^* \otimes \OP(-1) \ar[d]^{\alpha}\ar[r]&V \otimes \Omega^1_{\PP^2}(1) \ar[d] \ar[r]& I^* \otimes \OP \ar@{=}[d]&\\
0\ar[r]& K \otimes \OP \ar[d]^{\beta}\ar[r]&V \otimes U \otimes \OP \ar[d]\ar[r]& I^* \otimes \OP\ar[d]\ar[r]& 0\\
0\ar[r]&V \otimes \OP(1)\ar@{=}[r]&V \otimes \OP(1)\ar[d] \ar[r]& 0&\\
&&0&&} \end{equation}

we get the monad:
\begin{equation}\label{monade duale}
	V^* \otimes \OP(-1) \rightarrow V \otimes \Omega^1_{\PP^2}(1) \rightarrow I^* \otimes \OP 
\end{equation}
which is nothing but the dual \eqref{monade}. The second row comes from the cohomology of the Euler sequence tensored by $E(-1)$:
$$0 \to E \otimes \Omega^1_{\PP^2} \to U \otimes E(-1) \to E \to 0,$$
that at the $\HH^1$ level reads:
$$0 \to K \to U \otimes V \to I^* \to 0,$$
On the converse, to go from \eqref{monade} to \eqref{monade lineare} 
it is enough for the bundle $E$ to be stable. 

%%%%%%%%%%%%%%%%%%%%%%%%%%%%%%%%%%%%%%%%%%%%%%%%%%%%%%%%%%%%%%%%%%%%%%%%%%%%%%%%%
\subsection{Explicit examples}
%%%%%%%%%%%%%%%%%%%%%%%%%%%%%%%%%%%%%%%%%%%%%%%%%%%%%%%%%%%%%%%%%%%%%%%%%%%%%%%%%

According to our construction, in order to obtain explicit examples of elements of $M_{ort}^0(r,n)$ we simply need to 
exhibit the morphism $f \in U \otimes \wedge^2 V$, given by its three skew-symmetric $n \times n$ slices. In other words we need two skew-symmetric matrices 
$A$ and $B$ such that $\rk(AJB-BJA)=r$. It seems interesting and useful to give such an explicit construction for the 
linear monad. For other explicit examples, see \cite{jardim-marchesi-wissdorff}.
 
The case when the rank $r$ equals the second Chern class $n$ is the easiest, because the monad \eqref{monade} 
simplifies, and its dual \eqref{monade duale} gives us a resolution of the bundle $E$:
\begin{equation}\label{monade=risoluzione}
0 \rightarrow \OP(-1) \otimes V^* \xrightarrow{\tr{f}} \Omega^1_{\PP^2}(1) \otimes V \rightarrow E \rightarrow 0
\end{equation}

Keeping the same notations as above, let $f \in U \otimes \wedge^2 V$ be given by its three skew-symmetric $n \times n$ slices $A$, $J$ and $B$. 
We want to construct explicitly the linear monad having $E$ as its cohomology. We can again combine everything in a diagram:

\begin{equation}
\xymatrix{
&&0\ar[d]&0\ar[d]&\\
0 \ar[r]& \OP(-1)^n \ar[r]^{\tr{f}}\ar@{=}[d]& \Omega^1_{\PP^2}(1)^n \ar[d]\ar[r]& E \ar[r]\ar[d]& 0\\
0\ar[r]& \OP(-1)^n \ar[r]^\alpha& \OP^{3n} \ar[d]^\beta\ar[r]& \Ker \beta \ar[d]\ar[r]& 0\\
&&\OP(1)^n\ar@{=}[r]\ar[d]&\OP(1)^n\ar[d]&\\
&&0&0&
} \end{equation}

To construct explicitly $\alpha$ and $\beta$ we start from the morphism $f$. Let $\C[x_0,x_1,x_2]$ be the coordinate ring. Denote by $\ul{x_i}:=x_i I_n$ the $n \times n$ 
scalar matrix, and set:
$$X:=\left[
\begin{array}{c}
\ul{x_0}\\
\ul{x_1}\\
\ul{x_2}
\end{array}
\right].
$$
Recall that the matrix
\begin{equation}
    \HH^0(f)=\left[
\begin{array}{ccc}
0&J&A\\
-J&0&B\\
-A&-B&0
\end{array}
\right]
\end{equation}
is symmetric, and in this case it is also of maximal rank $\rk \HH^0(f)=2n+r=3n$. Considering the associated complex quadratic form, it factorizes as 
$\HH^0(f)=\tr{Q}Q$ for some matrix $Q$. Now set $\alpha:= Q X$ and $\beta:=\tr{\alpha}$, so that:
$$\beta\alpha= (\tr{X} \tr{Q})(QX)=\tr{X}(\tr{Q}Q)X=\tr{X} \HH^0(f) X=0.$$

%%%%%%%%%%%%%%%%%%%%%%%%%%%%%%%%%%%%%%%%%%%%%%%%%%%%%%%%%%%%%%%%%%%%%%%%%%%%%%%%%
\subsection{Further questions and the case of odd $c_2$}
%%%%%%%%%%%%%%%%%%%%%%%%%%%%%%%%%%%%%%%%%%%%%%%%%%%%%%%%%%%%%%%%%%%%%%%%%%%%%%%%%
The work presented in this paper leaves many open questions. We are currently working in generalizing the proof of the main theorem 
to all values of $3 \le r \le n$. 

We believe that the case $c_2$ odd is also very interesting. We know that these bundles cannot have trivial splitting on the general line, and that 
they do not deform to ones that do. How can we study their moduli space? Notice that $(S^2\T\PP^2)(-3)$ is an example of a stable rank $3$ orthogonal 
bundle on $\PP^2$ with Chern classes $(c_1,c_2)=(0,3)$, and that its splitting type on the general line $\ell$ 
is $\OO_\ell(-1) \oplus \OO_\ell \oplus \OO_\ell(1)$.

\thispagestyle{empty}
\bibliographystyle{amsalpha}
\bibliography{biblioada}

\end{document}